\documentclass[12pt,a4paper]{article}
\usepackage[utf8]{inputenc}

\usepackage[left=2.5cm,right=2.3cm,top=2.5cm,bottom=2.5cm]{geometry}
\usepackage{amsmath,amssymb,amscd,amsfonts, amsthm,color,epsfig,
fancyhdr,graphicx,latexsym,mathrsfs,multicol,psfrag,setspace,wrapfig}
\usepackage[normalem]{ulem}
\usepackage[
  bookmarksopen,
  colorlinks,
  linkcolor = blue,
  urlcolor  = blue,
  citecolor = black,
  menucolor = blue
]{hyperref}
\usepackage{pdfpages}

\def\yd{y^\delta}
\def\argmin{\text{argmin}}

\newcommand\R{\mathbb{R}}

\newcommand\N{\mathbb{N}}

\def\argmin{\text{argmin}}
\newtheorem{theo}{Theorem}
\newtheorem{lem}{Lemma}
\newtheorem{cor}{Corollary}
\newtheorem{pr}{Proposition}
\newtheorem{df}{Definition}
\newtheorem{rem}{Remark}
\newtheorem{ex}{Example}
\newtheorem{ass}{Assumption}
\newtheorem{prob}{Problem}

\newcommand{\domainf}{{\mathcal{D}(F)}}
\newcommand{\xdag}{{x^\dagger}}
\newcommand{\xd}{{x^\delta_{m,\alpha}}}
\newcommand{\tx}{{\tilde{x}}}

\newcommand\he{H^{1+\varepsilon}(\mathbb{R}_+\times\mathbb{R})}

\begin{document}

\setcounter{footnote}{1}

\title{\bf On the Choice of the Tikhonov Regularization Parameter and the Discretization Level: A Discrepancy-Based Strategy}

\author{V. Albani\thanks{IMPA, Estr. D. Castorina
         110, 22460-320 Rio de Janeiro, Brazil, \href{mailto:vvla@impa.br}{\tt
         vvla@impa.br}}, \,
        A. De Cezaro\thanks{ IMEF/FURG,
         Av. Italia km 8, 96201-900 Rio  Grande, Brasil,\, \href{mailto:decezaro@impa.br}
        {\tt decezaro@impa.br}}\,
                  and \
        J. P.~Zubelli\thanks{IMPA, Estr. D. Castorina 110,
         22460-320 Rio de Janeiro, Brasil,\, \href{mailto:zubelli@impa.br}
        {\tt zubelli@impa.br}}
        }

\date{\today}

\maketitle

\begin{abstract}
We address the classical issue of appropriate choice of the regularization and discretization level for the Tikhonov regularization of an inverse problem with imperfectly measured data.
We focus on the fact that the proper choice of the discretization level in the domain together with the regularization parameter is a key feature in adequate regularization. 
We propose a discrepancy-based choice for these quantities by applying a relaxed version of Morozov's discrepancy principle.
Indeed, we prove the existence of the discretization level and the regularization parameter satisfying such discrepancy. We also prove associated regularizing properties concerning the Tikhonov minimizers. 

\end{abstract}

\section{Introduction}\label{sec:introduction}

In many applications, inverse problems are solved under a finite-dimensional and discrete setup with noisy, and sparse data, although the theoretical framework is infinite dimensional. 
See \cite{ern,groetsch,schervar}. Thus, the relation between the finite- and the infinite-dimensional descriptions of the same problem should be well-understood. 
More precisely, it is important to state a criterion to find appropriately the domain discretization level in terms of the available data, in order to find a reliable solution of the inverse problem, which is in general ill-posed.

Thus, under the context of Tikhonov-type regularization, we propose a discrepancy-based rule for choosing appropriately a regularization parameter and a domain  discretization level. We also establish the corresponding regularizing properties of this rule under fairly general assumptions inspired by \cite[Chapter~3]{schervar}.

Several authors have considered discretization as a regularization tool. See \cite{groetsch,gn,kns,ns} and references therein. We go one step further by analyzing the interplay between the regularization parameter and the domain discretization level from a discrepancy principle viewpoint.

Assume that a model is given by the operator 
$F:\domainf\subset X \rightarrow Y$, 
defined in the reflexive Banach spaces $X$ and $Y$, with a convex domain $\mathcal{D}(F)$. Thus, we want to identify the element $x$ in $\domainf\subset X$ that generated, through $F$, the data $y\in\mathcal{R}(F)\subset Y$. In other words, we have the following problem:

\begin{prob}
Given $y\in \mathcal{R}(F)$, find $x\in\domainf$ satisfying:

\begin{equation}
 F(x) = y.
 \label{ip1}
\end{equation}
\label{prob1}
\end{prob}

Problem~\ref{prob1} is an idealization, since it is intrinsically assumed that the data $y$ is perfectly measurable, i.e., there is no uncertainty when measuring $y$. However, in practice we have access only to noisy data in $Y$, denoted by $y^\delta$. Furthermore, in general the inverse of the forward operator is not continuous or not well-defined (ill-posed).

When the statistics of the noise is available and the way it corrupts the data is known, the noisy and the noiseless data are related by:
$$
y^\delta = h(y,e),
$$
where $e$ is the noise, which is given by some random variable and the function $h(\cdot,\cdot\cdot)$ states how the uncertainties corrupts the data. 
See \cite{somersalo}.

\begin{rem}
Let us assume that $X$ and $Y$ are Hilbert spaces \footnote{The case of Banach spaces follows similarly by replacing the scalar product by the dual pairing.}.

\noindent Let the noise be additive, i.e., $h(y,e) = y+e$, and let $e$ be a zero-mean Gaussian random variable with covariance operator $\Sigma$, where $\Sigma:Y\longrightarrow Y$ is a positive-definite and bounded from below linear operator. Then, the noiseless and the noisy data should satisfy
\begin{equation} 
\langle y-\yd,\Sigma^{-1}(y-\yd) \rangle \leq \delta^2,
\label{estdata2}
\end{equation}
with $\delta>0$ and $ \|\Sigma\| \geq 1/\delta^2$.

\noindent By the positiveness and the boundedness of $\Sigma$, it follows that $\langle \cdot,\Sigma^{-1}\cdot\cdot \rangle$ is a scalar product equivalent to the standard one of $Y$. In other words, $Y$ with $\langle \cdot,\Sigma^{-1}\cdot\cdot \rangle$ is also a Hilbert space. Therefore, when we have a zero-mean Gaussian noise with covariance operator $\Sigma$, there is no loss of generality if we assume that $\yd$, $y$ and the noise level $\delta>0$ satisfy:

\begin{equation}
 \|y - \yd\| \leq \delta.
\label{estdata}
\end{equation}
\end{rem}

The presence of noise and the intrinsic ill-posedness of Problem~\ref{prob1} imply that some regularization technique should be addressed to find stable approximate solutions. Thus, we investigate a Tikhonov-type regularization approach. See \cite{ern,schervar}. More precisely, we analyze:

\begin{prob}
Find a minimizer in the domain of the operator $F$, which is assumed to be convex, for the Tikhonov functional
\begin{equation}
 \mathcal{F}^{\yd}_{\alpha,x_0}(x) = \|F(x) - \yd\|^p + \alpha f_{x_0}(x),
\label{tik1}
\end{equation}
with $\alpha > 0$ and $1\leq p < +\infty$.
\label{prob2}
\end{prob}
In the sequel, we shall make some fairly general assumptions on the functional $f_{x_0}$. We remark that, we introduced $x_0$ in \eqref{tik1} to allow the introduction of some {\it a priori} information.

We analyze Problem~\ref{prob1} in a discrete setup. Thus, besides choosing appropriately the  regularization parameter $\alpha$, we should choose also a proper level of discretization in the domain of the forward operator. Such choice should be made taking into account the available data $y^\delta$.
Thus, we base our choice of both parameters on the same relaxed version of Morozov's discrepancy principle. 
The present approach applies in a nontrivial way the methodology developed in \cite{moro,bonesky,anram,anram2} to the context of nonlinear operators under a discrete setting. We also establish that the continuous case can be recovered from the discrete one, when the discretization level goes to infinity. 

We propose an {\em a posteriori} choice for the discretization level in the domain of the forward operator and the regularization parameter $\alpha$ in the Tikhonov functional \eqref{tik1} based on the following relaxed version of the Morozov's discrepancy principle:
\begin{prob}
Given $1 < \tau < \lambda$ fixed, find $m\in \N$ and $\alpha>0$, such that
\begin{equation}
 \tau\delta \leq \|F(\xd) - \yd\| \leq \lambda\delta,
\label{disc1}
\end{equation}
where $\xd$, is a minimizer of (\ref{tik1}) in $\domainf \cap X_m$, with $X_m$ a finite-dimensional subspace of $X$.
\end{prob}

Under this framework, we prove that, if Problem~\ref{prob1} has a unique solution and the parameters $\alpha$ and $m$ satisfy \eqref{disc1}, then its discrete regularized reconstructions (weakly) converge, when the noise level $\delta$ goes to zero, to the solution of Problem~\ref{prob1}. When uniqueness does not hold, we also have convergence for some solution of Problem~\ref{prob1}. However, in this case, the regularized solutions converge to some $f_{x_0}$-minimizing solution of  Problem~\ref{prob1} (see Definition~\ref{defsol}), under some restrictions on the choice of $m$.

We observe that, in general the set of elements in the finite-dimensional subspace $X_m$ satisfying (\ref{disc1}) maybe empty. Thus, we shall prove that there exists some $m$ and $\alpha$ such that the discrepancy principle (\ref{disc1}) is satisfied. 
Furthermore, under suitable assumptions, convergence and convergence-rate results for the Tikhonov minimizers can be obtained for such parameters in terms of the noise level $\delta$.

Part of the proof of these results rely on the well-posedness of a discrete version of the Morozov's discrepancy principle presented in \cite{anram,anram2,bonesky}. 

We also present some guidelines to adapt the proofs of these results when the forward operator is replaced by a discrete approximation or the discrepancy principle \eqref{disc1} is replaced by the sequential discrepancy principle studied in \cite{ahm}.

\noindent This article is divided as follows:

\noindent In Section~\ref{sec:preliminar} we introduce the discrete setup and make some assumptions concerning Tikhonov-type regularization. We also present existence and stability results for the minimizers of (\ref{tik1}) with a fixed discretization level. In addition, we prove the convergence of the approximated solutions to some solution of Problem~\ref{prob1}.
\\
The well-posedness of the discrepancy principle (\ref{disc1}) for finding the appropriate discretization level in the domain and regularization parameter is stated in Section~\ref{sec:discrepancy2}. 
\\
The convergence and convergence-rate results associated to the discrepancy principle (\ref{disc1}) are established in Section~\ref{sec:convergence}.
\\
We also observe a key change in the proof of the convergence result when, instead of the forward operator $F$, we consider a discrete approximation of $F$ in Section~\ref{sec:addendum}.
\\
In Section~\ref{sec:addendumb}, we introduce an alternative discrepancy principle, which is more general than that of Equation~\eqref{disc1}. We also comment the principal changes in the proof of the convergence results, when using this discrepancy principle.\\
Section~\ref{sec:numerics} is devoted to numerical examples that illustrate the present approach.
\\
The regularizing properties of the auxiliary discrepancy principle introduced in Section~\ref{sec:discrepancy2} is proved in Appendix~\ref{app:morozov}.\\

\section{Preliminaries}\label{sec:preliminar}
We now define the discrete framework used in the subsequent sections. We also present some preliminary results on Tikhonov regularization under the discrete setup and make some important assumptions.
\begin{ass}
The regularizing functional $f_{x_0} : \mathcal{D}(f_{x_0})\rightarrow \mathbb{R}_+$ is weakly lower semi-continuous, convex, coercive, and proper. We also assume that $\domainf$ is in the interior of $\mathcal{D}(f_{x_0})$.
\label{ass00}
\end{ass}

\begin{ass}
The forward operator $F$ is continuous under the strong topologies of $X$ and $Y$. We also assume that the level sets
$$
\mathcal{M}_{\alpha}(\rho) = \{ x \in \domainf ~:~ \mathcal{F}^{\yd}_{\alpha,x_0}(x) \leq \rho \}
$$
are weakly pre-compact and weakly closed. Moreover, the restriction of $F$ to $\mathcal{M}_{\alpha}(\rho)$ is weakly continuous under the weak topologies of $X$ and $Y$.
\label{ass0}
\end{ass}

\begin{df}
An element $\xdag$ of $\domainf$ is called a least-square $f_{x_0}$-minimizing solution or simply an $f_{x_0}$-minimizing solution of Problem~\ref{prob1} if it is a least-square solution, i.e.,
$$
\xdag \in \mathcal{LS} := \left\{x \in \domainf ~:~ \|F(x) - y\| = 0\right\}
$$
and minimizes $f_{x_0}$ in $\mathcal{LS}$, i.e.,
$$
\xdag \in \mathcal{L}:=\displaystyle\argmin\{f_{x_0}(x) ~:~ x\in\mathcal{LS}\}.
$$
\label{defsol}
\end{df}
\noindent We always assume that $\mathcal{L} \not=\emptyset$.

Note that the sets $\mathcal{LS}$ and $\mathcal{L}$ depend on the noiseless data $y$.

\begin{ass}
 Let $\xdag$ be an $f_{x_0}$-minimizing solution for Problem~\ref{prob1} and $x_0 \in \domainf$ be fixed. We assume that:
\begin{equation}
 \displaystyle\liminf_{t\rightarrow 0^+}\displaystyle\frac{\|F((1-t)\xdag + t x_0) - y\|^p}{t} = 0
\end{equation}
\label{ass3}
\end{ass}
Note that Assumption~\ref{ass3} 
is satisfied by many classes of operators, such as the class of locally H\"older continuous functions with exponent greater than $1/2$, with $p = 2$. See \cite{ern,schervar} and references therein.

In the remaining part of this section we define the discrete setup of the present article.

\noindent Let the sequence $\{X_m\}_{m\in\N}$ of finite-dimensional subspaces of $X$ satisfy: 
\begin{equation}
 X_{m} \subset X_{m+1}, ~\text{for } m \in \N, ~\text{ and }~ \overline{\displaystyle\bigcup_{m \in \N}X_m} = X.
\label{def1}
\end{equation}

\begin{df}
Define the finite-dimensional sets:
\begin{equation}
 \mathcal{D}_m = \domainf \cap X_m, ~\text{ for }~ m \in \N.
\label{def2}
\end{equation}
\end{df}
\noindent The set $\mathcal{D}_m$ is convex since it is the intersection of a subspace of $X$ with a convex set. 

Note that, if we had chosen $\mathcal{D}_m$ as the orthogonal projection of $\domainf$ onto the finite-dimensional subspace $X_m$, we could possibly have that 
$
\mathcal{D}_m \cap X-\domainf \not= \emptyset,
$ 
since $F$ is not necessarily linear and $\domainf$ is not necessarily a subspace of $X$. 
Therefore, this definition ensures that $\mathcal{D}_m \subset \domainf$ for every $m \in \N$.

For now on, we assume that $\mathcal{D}_m \not= \emptyset$, for every $m$. Thus, we want to find $\xd \in \mathcal{D}_m$ minimizing (\ref{tik1}), with $m$ and $\alpha$ appropriately chosen.

The analysis that follows depends on how fast the restriction of the operator $F$ to $\mathcal{D}_m$ converges to $F$ as $m\rightarrow \infty$. Thus, we have the following definition:

\begin{df}
Let $P_m : X \rightarrow \mathcal{D}_m$ be the projection of $X$ onto $\mathcal{D}_m$, $\xdag$ be a least-square $f_{x_0}$-minimizing solution of Problem~\ref{prob1}. 
Define:
\begin{eqnarray}
 \gamma_m := \|F(\xdag) - F(P_m \xdag)\| &\text{ and }&
\phi_m := \|\xdag - P_m \xdag\|.
\label{def3}
\end{eqnarray}
\end{df}

\begin{lem} 
For every $x \in \domainf$, $\|F(x) - F(P_m x)\| \rightarrow 0$
when $m\rightarrow \infty$.
\label{lem1}
\end{lem}
\begin{proof} From (\ref{def1}) it follows that $\|x-P_m x\| \rightarrow 0$ as $m\rightarrow \infty$ for every $x \in \domainf$. 
Since the operator $F$ is continuous, the assertion follows.
\end{proof}

\subsection*{Existence and Stability of Tikhonov minimizers}

We consider the following optimization problem:
\begin{prob}
Find an element of
\begin{equation}
 \text{argmin}\{\|F(x) - \yd\|^p + \alpha f_{x_0}(x)\}, ~~~\text{ subject to }~~~ x \in \mathcal{D}_m.
\end{equation}
\label{tik2}
\end{prob}
We present below some well-known results concerning the existence and the stability of the solutions of Problem~\ref{tik2}. See \cite[Proposition 2.3]{pro}.

\begin{theo}[Existence]
 Let $m\in \N$ and $\delta>0$ be fixed. 
Then, for any $\yd \in Y$, it follows that Problem~\ref{tik2} has a solution.
\label{existence}
\end{theo}

\begin{df}
For given data $\yd$, we call a solution of Problem~\ref{tik2} stable if for a strongly convergent sequence $\{y_k\}_{k\in\N} \subset Y$, with limit $\yd$, the corresponding sequence $\{x_k\}_{k\in\N}\subset X$ of solutions of Problem~\ref{tik2}, where $\yd$ is replaced by $y_k$ in the functional of Problem~\ref{tik2}, has a weakly convergent subsequence $\{x_{k_l}\}_{l\in\N}$, with limit $\tx$, a solution of Problem~\ref{tik2} with data $\yd$.
\label{def5}
\end{df}

\begin{theo}[Stability]
 For each $m \in \N$, the solutions of Problem~\ref{tik2} are stable in the sense of Definition~\ref{def5}. Moreover, the convergent subsequence $\{x_{k_l}\}_{l\in\N}$ with limit $\tx$ from Definition~\ref{def5} satisfies the limit $f_{x_0}(x_{k_l})\rightarrow f_{x_0}(\tx)$.
\label{stability}
\end{theo}

\subsection*{Convergence}
The following theorem shows that the finite-dimensional Tikhonov minimizers converge to some $f_{x_0}$-minimizing solution of Problem~(\ref{ip1}).
\begin{theo} 
 Assume that $\alpha = \alpha(\delta,\gamma_m)>0$ satisfies the limits:
\begin{equation}
 \displaystyle\lim_{\delta,\gamma_m\rightarrow 0}\alpha(\delta,\gamma_m) = 0 ~~~\text{ and }
 \displaystyle\lim_{\delta,\gamma_m\rightarrow 0}\frac{(\delta + \gamma_m)^p}{\alpha(\delta,\gamma_m)}.
\label{estlim}
\end{equation}
Let $\{x_k\}_{k\in\N}$ be a sequence of solutions of Problem~\ref{tik2} with $x_k = x^{\delta_k}_{m_k,\alpha_k}$ and $\delta_k,\gamma_{m_k}\rightarrow 0$ when $k\rightarrow \infty$. Then, it has a weakly convergent subsequence $\{x_{k_l}\}_{l\in\N}$ with weak limit $\xdag$, an $f_{x_0}$-minimizing solution of Problem~\ref{ip1} with $f_{x_0}(x_{k_l})\rightarrow f_{x_0}(\xdag)$.
\label{conv1}
\end{theo}

\begin{proof} Let us choose a sequence $\{x_k\}_{k\in\N}$ of solutions of Problem~\ref{tik2} with data $y^{\delta_k}$, where $x_k \in X_{m_k}$ and $\alpha = \alpha(\delta_k,\gamma_{m_k})$ satisfy (\ref{estlim}). The existence of this sequence follows by Theorem~\ref{existence}. Then, we have the estimates:
$$
\begin{array}{rcl}
 \|F(x_k) -y^{\delta_k}\|^p + \alpha_k f_{x_0}(x_k) &\leq& \|F(P_{m_k}x^\dagger) -y^{\delta_k}\|^p + \alpha_k f_{x_0}(P_{m_k}\xdag)\\
 &\leq & (\delta_k + \gamma_{m_k})^p + \alpha_k f_{x_0}(P_{m_k}\xdag).
\end{array}
$$
Recall that the level sets $\mathcal{M}_{\alpha}(\rho)$ are pre-compact and $X_m \subset X_{m+1}$ for every $m$. Then, the above estimate implies that $\{x_k\}_{k\in\N}$ is bounded and thus it has a weakly convergent subsequence, also denoted by $\{x_k\}_{k\in\N}$, with limit $\tx \in \domainf$. It also follows that
$$
f_{x_0}(x_k) \leq \displaystyle\frac{(\delta_k + \gamma_{m_k})^p}{\alpha_k} +  f_{x_0}(P_{m_k}\xdag)
$$
Then, $\displaystyle\limsup_{k\rightarrow\infty} f_{x_0}(x_k) \leq f_{x_0}(\xdag)$, since $f_{x_0}(P_{m_k}\xdag) \rightarrow f_{x_0}(\xdag)$.
On the other hand, we have the following estimate:
$$
\|F(x_k) -y^{\delta_k}\|^p \leq (\delta_k + \gamma_{m_k})^p + \alpha_k f_{x_0}(P_{m_k}\xdag).
$$
Recall that the functional $f_{x_0}$ and the norm of $Y$ are weakly lower semi-continuous. Then,  the weak continuity of $F$ implies that:
$$
\begin{array}{rcl}
 \|F(\tx)-y\| &\leq& \displaystyle\liminf_{k\rightarrow\infty} \|F(x_k)-y\| \leq \displaystyle\limsup_{k\rightarrow\infty} \|F(x_k)-y\|\\
 &\leq& \lim_{k\rightarrow\infty}(\delta_k + \gamma_{m_k})^p + \alpha_k f_{x_0}(P_{m_k}\xdag) = 0
\end{array}
$$
$$
\begin{array}{rcl}
 f_{x_0}(\tx) &\leq& \displaystyle\liminf_{k\rightarrow\infty} f_{x_0}(x_k) \leq f_{x_0}(\xdag).
\end{array}
$$
Therefore, $\tx$ is an $f_{x_0}$-minimizing solution of Problem~(\ref{ip1}).

\end{proof}

\section{The Discrepancy Principle}\label{sec:discrepancy2}
In this section we consider the simultaneous choice of the discretization level in the domain and the regularization parameter based on the same discrepancy principle. 

\begin{df}
 Let $\delta> 0$ and $\yd$ be fixed. For $\lambda >\tau > 1$, we choose the greatest $m \in \N$ and $\alpha > 0$, with $m = m(\delta,\yd)$ and $\alpha = \alpha(\delta,\yd)$, such that
\begin{equation}
 \tau \delta \leq \|F(\xd) - \yd\| \leq \lambda\delta,
\label{discm2}
\end{equation}
holds for $\xd$ a solution of (\ref{tik2}) with these same $m$ and $\alpha$.
\label{def7}
\end{df}

\begin{pr}
 There exist $m \in \N$ and $\alpha > 0$ satisfying \eqref{discm2}.
\label{pro_disc}
\end{pr}
In Section~\ref{sec:addendumb} we present an alternative discrepancy principle, where only one of the inequalities of \eqref{discm2} should be satisfied. Thus, even the functional $\alpha \mapsto \|F(x^\delta_{m,\alpha} - y^\delta\|$ has some discontinuity, the inequality is satisfied by some $\alpha$, since, we shall see that $\lim_{\alpha \searrow 0}\|F(x^\delta_{m,\alpha} - y^\delta\| = 0$.

The existence of the regularization parameter and the discretization level satisfying the Discrepancy Principle~(\ref{discm2}) follows by the well-posedness of the modified Morozov's discrepancy principle \eqref{def6}. More precisely, we have to choose $m\in\N$ such that $\gamma_m$ satisfies a modified version of \eqref{discm2}. For this same $m$, we choose $\alpha > 0$ through \eqref{def6}, given that it is well-posed. Then, these $\alpha$ and $m$ satisfy the same discrepancy principle, as required. The well-posedness proof of this problem, for each discretization level $m\in\N$, is the aim of the following paragraph.
  
 In what follows we assume that $x_0$ in the penalization $f_{x_0}$ is an element of the finite-dimensional sub-domain $\mathcal{D}_{m_0}$ for some $m_0 \leq m$,  for every $m$ considered in the analysis.

\subsection*{Discrete Morozov's Principle}\label{sec:morozov}
We now present a criterion to choose the regularization parameter $\alpha$ by a modified version of the Morozov's principle for a given discretization level $m\in\N$.

In what follows we assume that:
\begin{equation}
f_{x_0} (x) = 0 \text{ if, and only if, } x = x_0.
\label{condm} 
\end{equation}

\begin{df}
 Let $\delta$, $\yd$ and the domain discretization level $m$ be fixed. For $\alpha \in \R^+$, we define the functionals:
\begin{eqnarray}
 L(\xd) &=& \|F(\xd) - \yd\|, \label{func1}\\
 H(\xd) &=& f_{x_0}(\xd), \label{func2}\\
 I(\alpha) &=& \|F(\xd)-\yd\|^p + \alpha f_{x_0}(\xd). \label{func3}
\end{eqnarray}
We also define the set of all solutions of Problem~\ref{tik2} for each $\alpha \in (0,\infty)$ and $m \in \N$:
 $$
 M_{\alpha,m} := \{\xd \in \mathcal{D}_m ~:~ \mathcal{F}^{\yd}_{\alpha,x_0}(\xd) \leq \mathcal{F}^{\yd}_{\alpha,x_0}(x), \forall x \in X_m\}.
 $$
 Note that, in what follows we assume that the Tikhonov functional is defined for any $x \in X_m$, i.e., it assume finite values if $x\in \mathcal{D}_m$ and it assumes the value $+\infty$ if $x\not\in \mathcal{D}_m$. Note also that the definition of the sets $M_{\alpha,m}$ implies that every local minimizer of the Tikhonov functional is a global minimizer in $X_m$.
\label{def8}
\end{df}

In the following results, some properties of the functionals $L$, $H$ and $I$ are presented. See \cite[Section~{2.6}]{tikar}.

\begin{lem}{\cite[Lemma 2.6.1]{tikar}}
As functions of $\alpha \in (0,\infty)$, it follows that, the functional $H(\cdot)$ is non-increasing and the functionals $L$ and $I$ are non-decreasing. More precisely, for $0 < \alpha < \beta$ we have
$$
\displaystyle\sup_{x \in M_{m,\alpha}}L(x) \leq \displaystyle\inf_{x \in M_{m,\beta}}L(x),~~
\displaystyle\inf_{x \in M_{m,\alpha}}H(x) \geq \displaystyle\sup_{x \in M_{m,\beta}}H(x) ~~\text{ and }~~
I(\alpha) \leq I(\beta).
$$
\label{lemoro1}
\end{lem}

\begin{lem}
The functional $I:(0,\infty)\rightarrow [0,\infty]$ is continuous. The sets
$$
M_m := \displaystyle\left\{\alpha > 0 \left| \displaystyle\inf_{x \in M_{m,\alpha}}L(x) < \displaystyle\sup_{x \in M_{m,\alpha}}L(x)\right.\right\}
$$
and
$$
N_m:= \displaystyle\left\{\alpha > 0 \left| \displaystyle\inf_{x \in M_{m,\alpha}}H(x) < \displaystyle\sup_{x \in M_{m,\alpha}}H(x)\right.\right\}
$$
are countable and coincide. Moreover, for each $m\in\N$ the maps $L$ and $H$ are continuous in $(0,\infty) \backslash M_m$.
\label{lemoro2}
\end{lem}

\begin{lem}
For each $\overline{\alpha}>0$, there exist $x_1,x_2 \in M_{m,\overline{\alpha}}$ such that
$$
L(x_1) = \displaystyle\inf_{x \in M_{m,\overline{\alpha}}}L(x) ~~~\text{and} ~~~ 
L(x_2) = \displaystyle\sup_{x \in M_{m,\overline{\alpha}}}L(x).
$$
\label{lemoro3}
\end{lem}

\begin{rem}
 Even we are under a discrete setting, it follows that the proofs of the Lemmas~\ref{lemoro1}, \ref{lemoro2} and \ref{lemoro3} hold, if the functionals $L$, $H$ and $I$ are restricted to the finite-dimensional subspace $X_m$, with $m \in \N$ fixed.
\end{rem}

In the present section we consider the following relaxed version of Morozov's discrepancy principle:
\begin{df}[Discrete Morozov's Principle]
 Let $\delta$, $\yd$ and the domain discretization level $m$ be fixed. Define $\tau_1:= \tau$ and let $\tau_2$ be such that $1 < \tau_1 \leq \tau_2 < \lambda$. Then, find $\alpha = \alpha(\delta,\yd,m) > 0$ such that
\begin{equation}
 \tau_1(\delta + \gamma_m) \leq \|F(\xd) - \yd\| \leq \tau_2(\delta + \gamma_m),
\end{equation}
holds for $\xd$, a solution of Problem~\ref{tik2}.
\label{def6}
\end{df}

In Section~\ref{sec:addendumb} we present an alternative discrepancy principle.

\begin{pr}
Let $1<\tau_1 \leq \tau_2$ be fixed. Suppose that $\|F(P_m x_0) - \yd\| > \tau_2(\delta+\gamma_m)$. Then, we can find $\underline{\alpha}$ and $\overline{\alpha}>0$, such that
$$
L(x_1) < \tau_1(\delta + \gamma_m) \leq \tau_2 (\delta + \gamma_m) < L(x_2),
$$
where we denote $x_1 :=  x^\delta_{m,\underline{\alpha}}$ and $x_2 := x^\delta_{m,\overline{\alpha}}$.
\label{pr7}
\end{pr}

\begin{proof}  Let $m$ be fixed and let the sequence $\{\alpha_k\}_{k\in\N}$ converge monotonically to zero. By Theorem~\ref{existence}, it is possible to find the sequence of  Tikhonov minimizers of Problem~\ref{tik2}, $\{x_k\}_{k\in\N}$, with $x_k = x^\delta_{m,\alpha_k}$. By the definition of the functionals $L$ and $I$, it follows that
$$
L(x_k) \leq  I(\alpha_k) \leq \|F(P_{m}\xdag) - \yd\|^p + \alpha_k f_{x_0}(P_{m}\xdag) \leq (\gamma_m + \delta)^p + \alpha_k f_{x_0}(P_{m}\xdag).
$$
Note that, $f_{x_0}(P_{m}\xdag)$ is fixed and $\tau_1 > 1$. It follows that, 
$
L(x_k) < \tau_1(\gamma_m + \delta)^p. 
$ 
Then, for a sufficiently large $\tilde{k}$ the above inequality holds. This allows us to set $\underline{\alpha} := \alpha_{\tilde{k}}$.

On the other hand, let us assume that the sequence  $\{\alpha_k\}_{k\in\N}$ satisfies $\lim_{k\rightarrow\infty}\alpha_k = +\infty$. Again, by Theorem~\ref{existence}, it is possible to choose a sequence of minimizers of Problem~\ref{tik2} $\{x_k\}_{k\in\N}$, with $x_k = x^\delta_{m,\alpha_k}$ for each $k\in\N$. Since $f_{x_0}(x) = 0$ if, and only if, $x=x_0$, it follows that,
$$
H(x_k) \leq \displaystyle\frac{1}{\alpha_k}\left(\|F(x_0) - \yd\|^p + \alpha_k f_{x_0}(x_0)\right) = \displaystyle\frac{1}{\alpha_k}\|F(x_0) - \yd\|^p\rightarrow 0.
$$
Then, $f_{x_0}(x_k)\rightarrow 0$ and $x_k \rightharpoonup x_0$. See \cite[Lemma~2.2]{anram}. By the weak continuity of $F$ and the weakly lower semi-continuity of the norm of $Y$, it follows that
$$
\|F(x_0) - \yd\| \leq \displaystyle\liminf_{k\rightarrow\infty} \|F(x_k) - \yd\|.
$$
This implies that there exists a sufficiently large $\hat{k}$ such that $\|F(x_{\hat{k}}) - \yd\| > \tau_2(\delta + \gamma_m)$.
Then, set $\overline{\alpha} := \alpha_{\hat{k}}$.

\end{proof}
Following \cite{anram} we have the following:
\begin{ass}
For each $m \in \N$ fixed, there is no $\alpha > 0$ such that the estimate
\begin{equation}
 \|F(\underline{x})-\yd\| < \tau_1(\delta + \gamma_m) \leq \tau_2(\delta + \gamma_m) < \|F(\overline{x})-\yd\|.
\label{cond3}
\end{equation}
is satisfied for some $\underline{x},\overline{x} \in M_{\alpha,m}$.
\label{ass4}
\end{ass}

The following result shows the well-posedness of the discrepancy principle of Definition~\ref{def6}.
\begin{theo}
Let Assumption~\ref{ass00}, \ref{ass0} and \ref{ass3} hold. Then, by Proposition~\ref{pr7} there exists an $\alpha := \alpha(\delta,\yd,\gamma_m)>0$ and an $\xd \in M_{m,\alpha}$, such that
\begin{equation}
\tau_1(\delta + \gamma_m) \leq \|F(\xd) - \yd\| \leq \tau_2(\delta + \gamma_m).
\label{dp1}
\end{equation}
\label{theo:morozov1}
\end{theo}

\begin{proof}
Let us assume that there exists no $\alpha > 0$ satisfying (\ref{dp1}). Define the sets:
$$
\begin{array}{c}
 A = \{ \alpha > 0 ~:~ \|F(\xd) - \yd\| < \tau_1(\delta + \gamma_m)~ \text{ for some } \xd \in M_{m,\alpha}\}\\
 B = \{ \alpha > 0 ~:~ \|F(\xd) - \yd\| > \tau_2(\delta + \gamma_m)~ \text{ for some } \xd \in M_{m,\alpha}\}
\end{array}
$$
We claim that the set $A$ is nonempty. Indeed, if $\alpha$ is sufficiently small, $\alpha f_{x_0}(P_m\xdag) < (\tau_1^p-1)(\delta + \gamma_m)^p$. Hence,
$$
\|F(\xd) - \yd\|^p+\alpha f_{x_0}(\xd) \leq \|F(P_m\xdag) - \yd\|^p + \alpha f_{x_0}(P_m\xdag) \leq \tau_1^p(\delta + \gamma_m).
$$
Note that, if $\alpha \in A$, then the condition $\|F(\xd) - \yd\| < \tau_1(\delta + \gamma_m)$ has to be satisfied for every $\xd \in M_{m,\alpha}$. Otherwise, (\ref{dp1}) or (\ref{cond3}) could be satisfied. It follows that, the similar assertion is true for $\alpha \in B$.

From Equation~(\ref{cond3}) it follows that $A\cap B = \emptyset$. Since we are assuming that there is no $\alpha > 0$ satisfying (\ref{dp1}), it follows that $A\cup B = \R_+$.

Define $\overline{\alpha} := \sup A$. By Proposition~\ref{pr7} and since $L$ is non-decreasing, it follows that $\overline{\alpha}<+\infty$. Then, it must belongs to $A$ or $B$.

If $\overline{\alpha} \in A$, then, it is possible to find a sequence $\{\alpha_k\}_{k\in\N} \subset B$ converging to $\overline{\alpha}$ with $\alpha_k > \overline{\alpha}$, since, $A\cup B = \R_+$ and $L$ is non-decreasing. Thus, let us select a sequence of minimizers $\{x_k\}_{k\in\N}$ with $x_k = x^\delta_{m,\alpha_k} \in M_{m,\alpha_k}$. 
We have the estimates:
$$
\tau_2(\delta+\gamma_m) < \displaystyle\lim_{k\rightarrow\infty}\|F(x_k) - \yd\| = \sup_{x\in M_{m,\overline{\alpha}}}\|F(x) - \yd\| < \tau_1 (\delta+\gamma_m).
$$
This is a contradiction, since $\tau_1 < \tau_2$.

On the other hand, assume that $\overline{\alpha}\in B$. Then, it is possible to find a sequence $\{\alpha_k\}_{k\in\N} \subset A$ converging to $\overline{\alpha}$, with $\alpha_k < \overline{\alpha}$. This follows by the same argument mentioned above. By selecting a sequence of minimizers $\{x_k\}_{k\in\N}$ with $x_k = x^\delta_{m,\alpha_k} \in M_{m,\alpha_k}$, we have the following estimates:
$$
\tau_2(\delta+\gamma_m) < \displaystyle\inf_{x\in M_{m,\overline{\alpha}}}\|F(x) -\yd\| = \displaystyle\lim_{k\rightarrow\infty}\|F(x_k) -\yd\| \leq \tau_1(\delta+\gamma_m).
$$

\end{proof}

We have thus established the well-posedness of the discrete Morozov's principle. See Appendix~\ref{app:morozov} for the corresponding regularizing properties.

Under the present setup, if we choose $m\in \N$ sufficiently large and such that
\begin{equation}
\gamma_m \leq \left(\displaystyle\frac{\lambda}{\tau_2} - 1 \right)\delta
\label{def4}
\end{equation}
is satisfied with $\lambda >\tau_2 > 1$.
Then, for this same $m\in\N$, it follows that, when $\alpha$ is chosen through Definition~\ref{def6}, the discrepancy
\begin{equation}
 \tau_1 \delta \leq \|F(\xd) - \yd\| \leq \lambda\delta,
\label{discm}
\end{equation}
is satisfied with $\xd$ a solution of (\ref{tik2}). This follows since, $\tau_1\delta \leq \tau_1(\delta + \gamma_m)$ and $\tau_2(\delta+\gamma_m) \leq \lambda \delta$.

This leads us to the proof of Proposition~\ref{pro_disc}:

\begin{proof} Let us consider the sets $M_{\alpha,m}$ of solutions of Problem~\ref{tik2}, corresponding to $\alpha$ and $m$. Recall that, by Theorem~\ref{existence}, these sets are nonempty. We also define the sets:
$$
A_{\delta,m} := \{ x\in \mathcal{D}_m ~:~ \tau \delta \leq \|F(x) - \yd\| \leq \lambda\delta \}.
$$
Note that, it may occur that $M_{\alpha,m}\cap A_{\delta,m} = \emptyset$. Thus, assuming that $\gamma_m = \mathcal{O}(\delta)$, by Theorem~\ref{theo:morozov1} there exist some $\alpha>0$ and $m\in \N$ such that $M_{\alpha,m}\cap A_{\delta,m} \not= \emptyset$. More precisely, 
let $m$ satisfy 
$\gamma_m \leq (\lambda/\tau_2 - 1)\delta$, with $\tau < \tau_2 < \lambda$. Also, let $\alpha$ be chosen through Definition~\ref{def6}, with $\tau_1 = \tau$. Then, for such $m$ and $\alpha$, it follows that
$$
\tau \delta \leq \tau(\delta + \gamma_m) \leq \|F(\xd) - \yd\| \leq \tau_2(\delta + \gamma_m) \leq \lambda \delta.
$$
Thus, $M_{\alpha,m}\cap A_{\delta,m} \not= \emptyset$.
\end{proof}

\begin{rem}
Therefore, the problem of finding $m$ and $\alpha$ through the discrepancy principle of Definition~\ref{def7} is well-posed. 
\end{rem}

\section{Regularizing Properties}\label{sec:convergence}
In the previous section we have established the well-posedness of the discrepancy principle \eqref{discm2} as a rule to select the parameters $m$ and $\alpha$, with fixed data $\yd$ and noise level $\delta$. We now explore some corresponding regularizing properties. 

\begin{df}
Let $\varepsilon \in (0,\tau-1)$ be fixed. Then, for every $m \in \N$, define the sets:
$$
H_m := \{x \in \mathcal{D}_m ~:~ \|F(x) - \yd \| < (\tau-\varepsilon)\delta\},
$$
with the same $\tau$ of Definition~\ref{def7}.
\end{df}

The sets defined above shall be used in the proof of convergence results, whenever we need to assume that $H_m$ is nonempty. Observe that, for $m$ sufficiently large, $H_m$ is indeed nonempty, since:
$$
H_m = \left(F^{-1}\left(B(\yd,(\tau-\varepsilon)\delta)\right)\cap\domainf \right)\cap X_m,
$$
where $B(\yd,(\tau-\varepsilon)\delta)$ is the open ball centered in $\yd$ and with radius $(\tau-\varepsilon)\delta$. 
It also follows that 
$$
\xdag \in \displaystyle\overline{\bigcup_{m\in\N} H_m}.
$$
Then, it is possible to find a sequence $\{\overline{x_k}\}_{k\in\N}$, with $\overline{x_k}\in H_{m_k}$, converging strongly to $\xdag$.

The following proposition states a connection between the discrete setting and the continuous one. 

\begin{pr} 
Let us consider the limit $m\rightarrow \infty$, with $\delta>0$ fixed. We select sequences $\{\alpha_k\}_{k\in\N}$ and $\{x^\delta_{m_k,\alpha_k}\}_{k\in\N}$, such that $\alpha_k\rightarrow\widetilde{\alpha}$ and $x^\delta_{m_k,\alpha_k}\rightharpoonup \widetilde{x}$, where, for each $k\in\N$, $x^\delta_{m_k,\alpha_k}$ is a solution of Problem~\ref{tik2} in $\mathcal{D}_{m_k}$ and $\alpha_k$ is the corresponding regularization parameter. Assume that, for each $k$, $x^\delta_{m_k,\alpha_k}$ satisfies the discrepancy principle \eqref{discm2}. Then, $\widetilde{x}$ is a Tikhonov minimizer in $\domainf$ with regularization parameter $\widetilde{\alpha}$, satisfying the discrepancy principle
\begin{equation}
 \tau \delta \leq \|F(\widetilde{x}) -\yd\| \leq \lambda \delta.
\label{morozov_cont}
\end{equation}
\label{pr:conv}
\end{pr}

\begin{proof} Choose a sufficiently large $m_0\in\N$, such that, for every $m\geq m_0$, the set $H_m$ is nonempty. Then, whenever $m \geq m_0$, we choose $\alpha = \alpha(\delta,\yd,m)$ satisfying \eqref{morozov_cont}, with $x^\delta_{m,\alpha}$ a corresponding Tikhonov solution in $\mathcal{D}_{m}$.

By Lemma~\ref{lemoro2}, the functionals defined in the Equations \eqref{func1}, \eqref{func2} and \eqref{func3} are monotone. It follows that $\displaystyle\liminf_{m\rightarrow\infty}\alpha(\delta,\yd,m)$ and $\displaystyle \limsup_{m\rightarrow\infty} \alpha(\delta,\yd,m)$ have finite values.

On the other hand, since the level sets of the Tikhonov functional are weakly pre-compact, it follows that the sequence $\{x^\delta_{m,\alpha_m}\}_{m\in\N}$, of corresponding Tikhonov minimizers, is weakly pre-compact.

Let us define $\widetilde{\alpha} = \displaystyle\liminf_{m\rightarrow\infty}\alpha_m$. We select the convergent subsequences $\{\alpha_{m_k}\}_{k\in\N}$ and $\{x^\delta_{m_k,\alpha_k}\}_{k\in\N}$, such that $\alpha_{m_k}\rightarrow\widetilde{\alpha}$ and $x^\delta_{m_k,\alpha_k}\rightharpoonup \widetilde{x}$, where $\widetilde{x} \in \domainf$. The convergence follows since $\domainf$ is weakly closed.

Recall that $F$ is weakly continuous. Then, we have the estimates:
$$
\tau \delta \leq \displaystyle\liminf_{k\rightarrow\infty}\|F(x^\delta_{m_k,\alpha_k})-\yd\|\leq \displaystyle\limsup_{k\rightarrow\infty}\|F(x^\delta_{m_k,\alpha_k})-\yd\|\leq \lambda \delta.
$$
This leads to $ \tau \delta \leq \|F(\widetilde{x}) -\yd\| \leq \lambda \delta$.

We now claim that $\widetilde{x}$ is a Tikhonov minimizer in $\domainf$, with regularization parameter $\widetilde{\alpha}$. Indeed, for an arbitrary and fixed $x\in\domainf$, we choose a sequence $\{x_k\}_{k\in\N}$, with $x_k\in\mathcal{D}_{m_k}$ for each $k \in \N$ and $x_k\rightarrow x$ strongly.  
Since $\domainf$ is in the interior of $\mathcal{D}(f_{x_0})$ and the operator $F$ is continuous, it follows that
\begin{equation}
\displaystyle\liminf_{k\rightarrow\infty}\mathcal{F}^\delta_{\alpha_{m_k},x_0}(x_k) =
\displaystyle\lim_{k\rightarrow\infty}\mathcal{F}^\delta_{\alpha_{m_k},x_0}(x_k) = \mathcal{F}^\delta_{\widetilde{\alpha},x_0}(x).
\label{proofeq1}
\end{equation}

\noindent By the (weak) lower semi-continuity of $f_{x_0}(\cdot)$ and $\|F(\cdot) - \yd\|$, it follows that:
\begin{equation}
 \|F(\widetilde{x}) - \yd\|^p + \widetilde{\alpha}f_{x_0}(\widetilde{x}) \leq
\displaystyle\liminf_{k\rightarrow\infty}\left\{\|F(x^\delta_{m_k,\alpha_{m_k}}) - \yd\|^p + \alpha_{m_k}f_{x_0}(x^\delta_{m_k,\alpha_{m_k}})\right\}.
\label{proofeq2}
\end{equation}

\noindent Since, for each $m$, $x^\delta_{m_k,\alpha_{m_k}}$ is a Tikhonov minimizer in $\mathcal{D}_{m_k}$, with regularization parameter $\alpha_{m_k}$. Thus, $\mathcal{F}^\delta_{\alpha_{m_k},x_0}(x^\delta_{m_k,\alpha_{m_k}})\leq \mathcal{F}^\delta_{\alpha_{m_k},x_0}(x_k)$, for every $k$. Therefore, by applying $\liminf$ on both sides and considering Equations~\eqref{proofeq1} and \eqref{proofeq2}, it follows that:
$$
 \|F(\widetilde{x}) - \yd\|^p + \widetilde{\alpha}f_{x_0}(\widetilde{x}) \leq  \|F(x) - \yd\|^p + \widetilde{\alpha}f_{x_0}(x).
$$
Since $x$ was arbitrarily chosen in $\domainf$, the assertion follows.
\end{proof}

\subsection*{Convergence}
We now present results concerning the convergence of the approximate solutions.

\begin{theo}
 Let $m$ and $\alpha$ satisfy the discrepancy principle \eqref{discm2}. Let us consider the sequence of real numbers $\{\delta_k\}_{k\in\N}$, satisfying $\delta_k>0$ and $\delta_k \rightarrow 0$. Then, every sequence of regularized solutions $\{x_k\}_{k\in\N}$, with $x_k = x^{\delta_k}_{m_k,\alpha_k}$, has a subsequence converging weakly to some least-square solution of Problem~\ref{prob1}. Moreover, If there exists a unique solution $\xdag$ for Problem~\ref{prob1}, then the whole sequence converges weakly to $\xdag$.
\label{prop:conv1}
\end{theo}
\begin{proof} Let us choose the sequence $\{\delta_k\}_{k\in\N}$, satisfying $\delta_k \rightarrow 0$ monotonically. For each $\delta_k$, it follows by Proposition~\ref{pro_disc} that there exist $m_k$ and $\alpha_k$ and some regularized solution $x_k = x^{\delta_k}_{m_k,\alpha_k}$ satisfying the discrepancy principle \eqref{discm2}. Then, we can find a sequence $\{x_k\}_{k\in\N}$, associated to $\{\delta_k\}_{k\in\N}$. 

According to Assumption~\ref{ass0}, the level sets of the Tikhonov functional~\eqref{tik2} are weakly pre-compact. Since $\domainf$ is convex, the sequence $\{x_k\}_{k\in\N}$ has a weakly convergent subsequence $\{x_{k_l}\}_{k\in\N}$ with limit $\tx \in \domainf$.

By the weak lower semi-continuity of the norm and the weak continuity of $F$, it follows that:
\begin{equation}
\|F(\tx)-y\| \leq \displaystyle\liminf_{l\rightarrow \infty} \|F(x_{k_l})-y^{\delta_{k_l}}\| + \delta_{k_l}.
\label{est:ws}
\end{equation}
Since, for each $l\in \N$, $x_{k_l}$ satisfies the discrepancy principle, in particular, $\|F(x_{k_l}) - y^{\delta_{k_l}}\| \leq \lambda \delta_{k_l}$, then:
\begin{equation}
\displaystyle\liminf_{l\rightarrow \infty} \|F(x_{k_l})-y^{\delta_{k_l}}\| + \delta_{k_l} \leq
\displaystyle\lim_{l\rightarrow\infty}(\lambda + 1)\delta_{k_l} = 0.
\label{est:ws2}
\end{equation}
This leads to 
$\|F(\tx)-y\| = 0.$ 
Therefore, $\tx$ is a least-square solution of Problem~\ref{prob1}. 
If the inverse problem has a unique solution $\xdag$, then $\xdag=\tx$ and $\gamma_{m_{k_l}} \rightarrow 0$. 
Furthermore, the whole sequence $\{x_k\}_{k\in \N}$ converges weakly to $\tx$, since it is the unique cluster point of $\{x_k\}_{k\in \N}$, which is bounded.
\end{proof}
\begin{rem}
 The existence of the parameters $m$ and $\alpha$ satisfying the discrepancy principle~\eqref{discm2} is guaranteed by Proposition~\ref{pro_disc}. Intuitively, $m$ is the largest discretization level satisfying such discrepancy principle and $\alpha$ is the associated Morozov's regularization parameter. When implementing the Tikhonov regularization numerically, the discrepancy principle can be used as a stopping criterion in the minimization procedure. 

Observe also that, if we are looking for a least-square solution of Problem~\ref{prob1}, or if the inverse problem has a unique solution, then no further assumption or restriction on the choice of $m$ is needed. Thus, the convergence result holds.
\end{rem}

\begin{theo}[Convergence]
Let $m$ and $\alpha$ satisfy the discrepancy principle~\eqref{discm2}. Les us consider  the sequence of real numbers $\{\delta_k\}_{k\in\N}$, satisfying $\delta_k>0$ and $\delta_k \rightarrow 0$. If $H_m$ is nonempty for every $m$, then, every sequence of regularized solutions $\{x_k\}_{k\in\N}$, with $x_k = x^{\delta_k}_{m_k,\alpha_k}$, has a subsequence converging weakly to a $f_{x_0}$-minimizing solution of Problem~\ref{prob1}. Moreover, the following limits hold:
\begin{equation}
\begin{array}{rcl}
 \displaystyle\lim_{\delta\rightarrow 0}\alpha(\delta,\yd) = 0,&
\text{and}
& \displaystyle\lim_{\delta\rightarrow 0}\frac{\delta^p}{\alpha(\delta,\yd,m(\delta,\yd))} = 0.
\end{array}
\label{limits}
\end{equation}
\label{theo:conv}
\end{theo}
\begin{proof} (i) Following the same arguments of the proof of Proposition~\ref{prop:conv1},  we choose the sequences $\{\delta_k\}_{k\in\N}$, with $\delta_k\downarrow 0$, and $\{x_k\}_{k\in\N}$, the corresponding sequence of Tikhonov minimizers, with $x_k = x^{\delta_k}_{m_k,\alpha_k}$. We can assume that the latter has a weak limit, $\tx \in \domainf$. By Equation~\eqref{est:ws}, $\tx$ is a least-square solution of Problem~\ref{prob1}. 
Let us assume also that each element of the sequence of regularization parameters $\{\alpha_k\}_{k\in\N}$, associated to $\{x_k\}_{k\in\N}$, satisfies the Morozov's principle.

Since $H_{m_k}\not=\emptyset$ for each $k\in \N$, it is possible to choose a sequence $\{\overline{x_k}\}_{k\in\N}$, such that $\overline{x_k}\in H_{m_k}$, for each $k\in \N$ and $\overline{x_k}\rightarrow \xdag$. Moreover, the following estimates hold:
$$
\begin{array}{rcl}
 \tau^p\delta_k^p + \alpha_k f_{x_0}(x_k) &\leq& \|F(x_k) - y^{\delta_k} \|^p + \alpha_k f_{x_0}(x_k)\\
 &\leq& \|F(\overline{x_k}) - y^{\delta_k} \|^p + \alpha_k f_{x_0}(\overline{x_k})\\
 &\leq& (\tau-\varepsilon)^p\delta_k^p + \alpha_k f_{x_0}(\overline{x_k}).
\end{array}
$$
The above estimates imply that,
\begin{equation}
 0 < \displaystyle\frac{\varepsilon^p\delta_k^p}{\alpha_k} \leq f_{x_0}(\overline{x_k}) - f_{x_0}(x_k),
\label{eq:comp}
\end{equation}
Thus, by the weak lower semi-continuity of $f_{x_0}$ and the above estimates, it follows that
\begin{equation}
f_{x_0}(\tx) \leq \displaystyle\liminf_{k\rightarrow \infty} f_{x_0}(x_k) \leq \displaystyle\liminf_{k\rightarrow \infty} f_{x_0}(\overline{x_k}) = f_{x_0}(\xdag).
\label{ineqsol} 
\end{equation}
Since $\tx$ is a solution of the Inverse Problem~\ref{prob1}, it follows from \eqref{ineqsol} that $\tx$ is also an $f_{x_0}$-minimizing solution. Hence, $f_{x_0}(\tx) = f_{x_0}(\xdag)$. 
Therefore, the inequalities in \eqref{eq:comp} imply that $f_{x_0}(\overline{x_k}) - f_{x_0}(x_k) \rightarrow 0$. Then, the second limit in \eqref{limits} holds.
\vspace{10pt}

\noindent(ii) Under the same hypotheses and notation of (i), assume that there exist a constant $c>0$ and a subsequence $\{\alpha_{k_l}\}_{l\in\N}$, with $\alpha_l = \alpha(\delta_{k_l},y^{\delta_{k_l}},m_{k_l})$, such that $\alpha_{k_l} > c$, for every $l\in\N$. 
Define $\underline{\alpha}:= \displaystyle\liminf_{l\rightarrow \infty} \alpha_{k_l}$.

\noindent Let us consider the corresponding subsequence of Tikhonov minimizers $\{x_{k_l}\}_{l\in\N}$. Since the original sequence is weakly convergent to $\tx$, it follows that the subsequence $\{x_{k_l}\}_{l\in\N}$ also weakly converges to $\tx$.

On the other hand, let us consider the subsequence $\{\overline{x_{k_l}}\}_{l\in\N}$, with $\overline{x_{k_l}}\rightarrow \xdag$ as in (i), we have the following estimates
$$
\begin{array}{rcl}
\underline{\alpha}f_{x_0}(\tx) = \displaystyle\liminf_{l\rightarrow\infty}[\alpha_{k_l}f_{x_0}(x_{k_l})] &\leq&
\displaystyle\liminf_{l\rightarrow\infty}\left(\|F(x_{k_l})-y^{\delta_{k_l}}\|^p + \alpha_{k_l}f_{x_0}(x_{k_l})\right)\\ &\leq& \displaystyle\liminf_{l\rightarrow\infty}\left(\|F(\overline{x_{k_l}})-y^{\delta_{k_l}}\|^p + \alpha_{k_l}f_{x_0}(\overline{x_{k_l}})\right)\\
&=&\underline{\alpha} f_{x_0}(\xdag).
\end{array}
$$
Since $\tx$ is the weak limit of $\{x_{k_l}\}_{l\in\N}$ and it is a $f_{x_0}$-minimizing solution of Problem~\ref{prob1}, it follows by the above estimates that
$$
\begin{array}{rcl}
 \|F(\tx) - y\|^p + \underline{\alpha} f_{x_0}(\tx) &\leq& 
\displaystyle\liminf_{l\rightarrow\infty}\left( \|F(x_{k_l})-y^{\delta_{k_l}}\|^p + \alpha_{k_l} f_{x_0}(x_{k_l}) \right)\\
&\leq& \displaystyle\liminf_{l\rightarrow\infty}\left(\|F(x) - y\|^p + \alpha_{k_l} f_{x_0}(x)\right)\\
&=& \|F(x) - y\|^p + \underline{\alpha} f_{x_0}(x), \text{ for every } x \in \domainf.
\end{array}
$$
Then, $\tx$ is a solution of Problem~\ref{tik2} with (noiseless) data $y$ and regularization parameter $\underline{\alpha}$.

\noindent Since $f_{x_0}$ is convex and $f_{x_0}(x_0) = 0$, the following estimate holds for every $t \in [0,1)$:
$$
f_{x_0}((1-t)\tx + t x_0) \leq (1-t)f_{x_0}(\tx) + t f_{x_0}(x_0) = (1-t)f_{x_0}(\tx).
$$
Thus, the following estimates hold:
$$
\begin{array}{rcl}
 \underline{\alpha}f_{x_0}(\tx) &\leq& \|F(\tx) - y\|^p + \underline{\alpha} f_{x_0}(\tx)\\
 &\leq & \|F((1-t)\tx + t x_0) - y\|^p + \underline{\alpha} f_{x_0}((1-t)\tx + t x_0)\\
 &\leq & \|F((1-t)\tx + t x_0) - y\|^p + \underline{\alpha} (1-t)f_{x_0}(\tx).
\end{array}
$$
This implies that,
$$
\underline{\alpha}tf_{x_0}(\tx)\leq \|F((1-t)\tx + t x_0) - y\|^p.
$$
Then, Assumption~\ref{ass3} implies that $f_{x_0}(\tx) = 0$, and thus $\tx = x_0$. This is a contradiction, since
$$
\|F(x_0)-y\| \geq \|F(x_0)-\yd\|-\|\yd-y\|\geq(\tau_1 - 1)\delta>0.
$$
Therefore, $\alpha_k = \alpha(\delta_k,y^{\delta_k},m_k) \rightarrow 0$.

\end{proof}

\begin{rem}
If we assume that $H_m$ is nonempty, then $\tx$, the weak limit of the sequence of minimizers defined in the proof of Theorem~\ref{theo:conv}, is an $f_{x_0}$-minimizing solution of Problem~\eqref{ip1}. 
Note that, it is always possible to increase the discretization level $m$, in order that $H_{m} \not=\emptyset$. 
On the other hand, if Problem~\eqref{ip1} has a unique solution, then this assumption is unnecessary. See Proposition~\ref{pr:conv}. 
\end{rem}

\subsection*{Convergence Rates}
The first theorem of the present section states the convergence rates of the regularized solutions of Problem~\ref{prob1}, associated to $m$ and $\alpha$ satisfying the discrepancy principle \eqref{discm2}, with respect to $\delta$. The following results generalize this theorem for more general forward operators, however further restrictions on the choice of $m$ are necessary. 

In the first part of this section we introduced some definitions, assumptions and auxiliary lemmas that are necessary to establish the convergence rates results.

\begin{df}[\cite{schervar},Definition~3.15]
Let $U$ denote a Banach space and 
$$f:\mathcal{D}(f) \subset U\rightarrow \R\cup\{\infty\}$$
be a convex functional with sub-differential $\partial f(u)$ at $u \in \mathcal{D}(f)$. The Bregman distance (or divergence) of $f$ at $u \in \mathcal{D}(f)$ and $\xi \in \partial f(u) \subset U^*$ is defined by
\begin{equation}
 D_{\xi}(\tilde{u},u) = f(\tilde{u}) - f(u) - \langle \xi, \tilde{u} - u\rangle,
\end{equation}
for every $\tilde{u} \in U$, where $\langle \cdot,\cdot\cdot\rangle$ is the dual product of $U^*$ and $U$. Moreover, the set
$$
\mathcal{D}_B(f) = \{x \in \mathcal{D}(f) ~: ~ \partial f(u)\not= \emptyset\}
$$
is called the Bregman domain of $f$.
\end{df}

\begin{lem}
 Let $m$ and $\alpha$ satisfy the discrepancy principle~\eqref{discm2}. Assume that for this $m$,  $H_{m}$ is nonempty and let $\xd$ be the respective minimizer of (\ref{tik1}). Then,
\begin{equation}
\begin{array}{rcl}
D_{\xi^\dagger}(\xd,\xdag) &\leq& D_{\xi^\dagger}(\overline{x_m},\xdag) + \langle \xi^\dagger, \xdag - \overline{x_m} \rangle + \langle \xi^\dagger, \xdag - \xd \rangle,
\end{array}
\end{equation}
for any $\xi^\dagger \in \partial f_{x_0}(\xdag)$.

 \label{lemma:convrates}
\end{lem}

\begin{proof}
\noindent Let $m=m(\delta,\yd)$ satisfy \eqref{discm2}. By the same arguments of the proof of Theorem~\ref{theo:conv} and assuming that $H_m\not=\emptyset$, it follows that $f_{x_0}(\xd) \leq f_{x_0}(\overline{x_m})$, where $\overline{x_m} \in H_m$ and $\overline{x_m}\rightarrow x^\dagger$ strongly.

\noindent For any $\xi^\dagger \in \partial f_{x_0}(\xdag)$, by Assumption~\ref{ass3} and the definition of Bregman distance, we have the following estimates:
$$
\begin{array}{rcl}
 D_{\xi^\dagger}(\xd,\xdag) &=& f_{x_0}(\xd) - f_{x_0}(\xdag) + \langle \xi^\dagger, \xd - \xdag \rangle \\
                          &\leq& f_{x_0}(\xd) \pm f_{x_0}(\overline{x_m}) - f_{x_0}(\xdag) \pm \langle \xi^\dagger, \overline{x_m} - \xdag \rangle - \langle \xi^\dagger, \xd - \xdag \rangle\\
			   &=& \underbrace{f_{x_0}(\xd) - f_{x_0}(\overline{x_m})}_{\leq 0} + D_{\xi^\dagger}(\overline{x_m},\xdag) + \langle \xi^\dagger, \xdag - \overline{x_m} \rangle + \langle \xi^\dagger, \xdag - \xd \rangle\\
			   &\leq& D_{\xi^\dagger}(\overline{x_m},\xdag) + \langle \xi^\dagger, \xdag - \overline{x_m} \rangle + \langle \xi^\dagger, \xdag - \xd \rangle.
\end{array}
$$
\end{proof}
\begin{lem}
Let $\alpha$ and $m$ be chosen through the discrepancy principle \eqref{discm2}.
Let also $\varepsilon^\prime \in (\varepsilon,\tau-1)$ be fixed and let $H_m$ be nonempty. Define
$$
\kappa := \inf \{f_{x_0}(x) ~:~ x \in \domainf \text{ and } \|F(x) - \yd\| \leq (\tau-\varepsilon^\prime)\delta \}.
$$
If $\displaystyle\inf_{x\in H_m} f_{x_0}(x) \leq  \kappa + \displaystyle\frac{\varepsilon^p\delta^p}{\alpha}$, then
\begin{equation}
 D_{\xi^\dagger}(\xd,\xdag) \leq \langle \xi^\dagger, \xdag - \xd \rangle.
\end{equation}
 \label{lemma:convrates2}
\end{lem}
\begin{proof}
By the discrepancy principle \eqref{discm2}, it follows that
\begin{multline}
\tau^p\delta^p + f_{x_0}(\xd) \leq \|F(\xd) - \yd\|^p +\alpha f_{x_0}(\xd) \\ \leq \|F(x) - \yd\|^p +\alpha f_{x_0}(x) \leq (\tau-\varepsilon)^p\delta^p + \alpha f_{x_0}(x),
\end{multline}
for every $x \in H_m$. Then,
$$
f_{x_0}(x) \geq f_{x_0}(\xd) + \displaystyle\frac{\varepsilon^p\delta^p}{\alpha}.
$$
It follows that $f_{x_0}(\xd) \leq \kappa \leq f_{x_0}(\xdag)$. Then, 
$$
D_{\xi^\dagger}(\xd,\xdag) = f_{x_0}(\xd) - f_{x_0}(\xdag) - \langle \xi^\dagger, \xd - \xdag \rangle \leq 
\langle \xi^\dagger, \xdag - \xd \rangle.
$$
\end{proof}

Note that, in Lemma~\ref{lemma:convrates2} we have assumed that $\varepsilon^\prime > \varepsilon$.  
Recall that $\varepsilon^p\delta^p/\alpha > 0$, $f_{x_0}$ is continuous in the interior of $\mathcal{D}(f_{x_0})$ and $\domainf$ is in the interior of $\mathcal{D}(f_{x_0})$. Then, the the estimate 
$$\displaystyle\inf_{x\in H_m} f_{x_0}(x) \leq  \kappa + \displaystyle\frac{\varepsilon^p\delta^p}{\alpha}$$
is satisfied for every sufficiently large $m\in \N$.

Inspired by \cite[Chapter~3]{schervar}, we have the following assumption:
\begin{ass}
 There exist $\beta_1 \in [0,1)$, $\beta_2 \geq 0$ and $\xi^\dagger \in \partial f_{x_0}(\xdag)$ such that
\begin{equation}
 \langle \xi^\dagger, \xdag - x\rangle \leq \beta_1 D_{\xi^\dagger}(x,\xdag) + \beta_2 \|F(x) - F(\xdag)\|
\end{equation}
for $x \in \mathcal{M}_{\alpha_{\max}}(\rho)$, where $\alpha_{\max},\rho > 0$ satisfy $\rho > \alpha_{\max}f_{x_0}(\xdag)$.
\label{ass2}
\end{ass}

\begin{theo}[Convergence Rates]
 Let $m$ and $\alpha$ be chosen through the discrepancy principle~\eqref{discm2} and let Assumption~\ref{ass2} be satisfied. In addition, suppose that $H_{m}$ is nonempty. If $\xd$ is a minimizer of (\ref{tik1}) and $\overline{x_m} \in H_m$, then the estimates hold:
\begin{equation}
\begin{array}{rcl}
\|F(\xd) - \yd\| \leq \lambda\delta &\text{ and }& D_{\xi^\dagger}(\xd,\xdag) \leq \displaystyle\frac{1+\beta_1}{1-\beta_1}D_{\xi^\dagger}(\overline{x_m},\xdag) 
+ \frac{\beta_2}{1-\beta_1}(\tau+\lambda+2)\delta,
\end{array}
\label{rate1}
\end{equation}
with $\xi^\dagger \in \partial f_{x_0}(\xdag)$.
Moreover, if the hypotheses of Lemma~\ref{lemma:convrates2} also hold, we have:
\begin{equation}
\begin{array}{rcl}
\|F(\xd) - \yd\| \leq \lambda\delta &\text{ and }& D_{\xi^\dagger}(\xd,\xdag) \leq \displaystyle\frac{\beta_2(1+\lambda)}{1-\beta_1}\delta.
\end{array}
\label{rate2}
\end{equation}
 \label{convrates:true}
\end{theo}

\begin{proof}
\noindent(i) The estimate
$$
\|F(\xd) - \yd\| = \mathcal{O}(\delta)
$$
follows directly by the discrepancy principle~\eqref{discm2}.

\noindent(ii) Let us prove the estimate~\eqref{rate1}. By Assumption~\ref{ass2}, if $\overline{x_m}$ is an element of $H_m$, then 
$$
\begin{array}{rcl}
\langle \xi^\dagger, \xdag - \overline{x_m} \rangle &\leq& \beta_1D_{\xi^\dagger}(\overline{x_m},\xdag) + \beta_2\|F(\overline{x_m})-F(\xdag)\|\\
 &\leq& \beta_1D_{\xi^\dagger}(\overline{x_m},\xdag) + \beta_2\|F(\overline{x_m})-\yd\| + \beta_2\delta\\
 &\leq& \beta_1D_{\xi^\dagger}(\overline{x_m},\xdag) + \beta_2(\tau+1)\delta.
\end{array}
$$
Analogously, it follows that
$$
\begin{array}{rcl}
\langle \xi^\dagger, \xdag - \xd \rangle &\leq& \beta_1D_{\xi^\dagger}(\xd,\xdag) + \beta_2(\lambda+1)\delta.
\end{array}
$$
Lemma~\ref{lemma:convrates} and the above estimates yield that
$$
\begin{array}{rcl}
D_{\xi^\dagger}(\xd,\xdag) &\leq& (\beta_1+1)D_{\xi^\dagger}(\overline{x_m},\xdag) + \beta_2(\tau+\lambda+2)\delta + \beta_1D_{\xi^\dagger}(\xd,\xdag).
\end{array}
$$
Thus,
$$
D_{\xi^\dagger}(\xd,\xdag) \leq \displaystyle\frac{1+\beta_1}{1-\beta_1}D_{\xi^\dagger}(\overline{x_m},\xdag) 
+ \frac{\beta_2}{1-\beta_1}(\tau+\lambda+2)\delta,
$$
and we have the second estimate in \eqref{rate1}.

\noindent(iii) We pass now to the proof of the second estimate in \eqref{rate2}, since the first one follows by the same arguments of the first estimate of \eqref{rate1}. By Lemma~\ref{lemma:convrates2} and Assumption~\ref{ass2}, there exist constants $\beta_1 \in [0,1)$ and $\beta_2 \geq 0$ such that
$$
D_{\xi^\dagger}(\xd,\xdag) \leq \beta_1 D_{\xi^\dagger}(\xd,\xdag) + \beta_2\|F(\xd)-F(\xdag)\|.
$$
Then,
$$
(1-\beta_1)D_{\xi^\dagger}(\xd,\xdag) \leq \beta_2\|F(\xd)-F(\xdag)\|.
$$
This implies that,
$$
D_{\xi^\dagger}(\xd,\xdag) \leq \displaystyle\frac{\beta_2}{1-\beta_1}\|F(\xd)-F(\xdag)\|.
$$
Since $\|F(\xd)-\yd\| \leq \lambda \delta$ and $\|F(\xdag) - \yd\| \leq \delta$ hold, it follows that $\|F(\xd) - F(\xdag)\| \leq (1+\lambda)\delta$. 
Then,
$$
D_{\xi^\dagger}(\xd,\xdag) \leq \displaystyle\frac{\beta_2(1+\lambda)}{1-\beta_1}\delta.
$$
\end{proof}

\begin{pr}
 Let $m$ and $\alpha$ satisfy the discrepancy principle~\eqref{discm2}. Assume that for the same $m$ and $\alpha$,  $H_{m}$ is nonempty and $\xd$ is a minimizer of \eqref{tik1}. Assume that $F$ is Frech\'et differentiable in $\xdag$ and let the source condition
\begin{equation}
 \exists ~\xi^\dagger \in \partial f_{x_0}(\xdag) \cap \mathcal{R}(F^\prime(\xdag)^*), \text{ {\em i.e.},} ~\exists~ \omega^\dagger \in Y^*\text{ s.t. } \xi^\dagger = F^\prime(\xdag)^*\omega^\dagger
\end{equation}
hold. Let also the estimates
\begin{equation}
 \|F^\prime(\xdag)(x - \xdag)\| \leq C\|F(x) - F(\xdag)\|
\end{equation}
hold with $C$ constant and $x$ in $B(\xdag,\eta)$, for some $\eta > 0$. Again, let $\overline{x_m}$ be an element of $H_m$. Then, we have the convergence rates
\begin{equation}
 \begin{array}{rcl}
\|F(\xd) - \yd\| \leq \lambda\delta &\text{ and }& D_{\xi^\dagger}(\xd,\xdag) = c\|\overline{x_m}-\xdag\| + C\|\omega^\dagger\|(2\tau + \lambda + 2)\delta. 
\end{array}
\end{equation}
\end{pr}

\begin{proof} The first estimate follows by the discrepancy principle~\eqref{discm2}. Let $\overline{x_m}$ be an element of $H_m$. 
Recall that $f_{x_0}$ is convex, then it is locally Lipschitz continuous in the interior of $\mathcal{D}(f_{x_0})$. Since $\xdag$ and $\overline{x_m}$ are indeed interior points of $\mathcal{D}(f_{x_0})$, then we have:
$$
\begin{array}{rcl}
D_{\xi^\dagger}(\overline{x_m},\xdag) &=& f_{x_0}(\overline{x_m}) - f_{x_0}(\xdag) - \langle \xi^\dagger, \overline{x_m} - \xdag\rangle\\
&\leq& c\|\overline{x_m}-\xdag\| + \|\omega^\dagger\|C\|F(\overline{x_m}) - F(\xdag)\|\\
 &\leq& c\|\overline{x_m}-\xdag\| + C\|\omega^\dagger\|(\tau+1)\delta.
\end{array}
$$
A similar argument yields that
$$
|\langle \xi^\dagger, \xd - \xdag \rangle| \leq C\|\omega^\dagger\|(\lambda+1)\delta.
$$
Then, from the above estimates and Lemma~\ref{lemma:convrates}, we have the following:
$$
\begin{array}{rcl}
 D_{\xi^\dagger}(\xd,\xdag) &\leq& D_{\xi^\dagger}(\overline{x_m},\xdag) + \langle \xi^\dagger, \xdag - \overline{x_m} \rangle + \langle \xi^\dagger, \xdag - \xd \rangle\\
 &\leq& c\|\overline{x_m}-\xdag\| + C\|\omega^\dagger\|(2\tau + \lambda + 2)\delta.
\end{array}
$$

\end{proof}

\begin{df}[q-Coerciveness]
 Let $1 \leq q < \infty$ and $u \in \mathcal{D}(f)$ be fixed. The Bregman distance $D_{\xi}(\cdot,u)$ is called $q$-coercive with constant $\zeta >0$, if the inequality
$$
D_{\xi}(\tilde{u},u) \geq \zeta\|\tilde{u} - u\|^q_{U}
$$
is satisfied for every $\tilde{u} \in \mathcal{D}(f)$.
\label{qcoer}
\end{df}

\begin{ex}
 Let $X$ be a Hilbert space and let $f_{x_0}(x) = \|x - x_0\|^2_X$ be the quadratic Tikhonov functional. It follows that the norm of $X$ is $2$-coercive, since:
$$
D_{\xi}(x,\tx) = 2\|x-\tx\|^2_X.
$$
Then, the estimate \eqref{rate2} of Theorem~\ref{convrates} implies in $L^2$-convergence with order $\mathcal{O}(\sqrt{\delta})$. See \cite{schervar}.
\label{ex1}
\end{ex}

\begin{ex}
Assume that $X = L^p(D)$ with $D \subset \R^n$ open and bounded. Assume also that $\mathcal{D}(f) = L^1_{>0}(D)$, i.e., the set of strictly positive $L^1(D)$ functions. 
If
$$
f_{x_0}(x) = \displaystyle\int_D[\log(x(s)/x_0(s)) - (x_0(s) - x(s))]ds
$$
is the Kullback-Leibler divergence, then Theorem~\ref{convrates} implies in the $L^1$-convergence. See \cite{resa}.
\label{ex2}
\end{ex}

\begin{rem}
In addition to the hypotheses of Theorem~\ref{theo:conv}, let us assume that the norm of $X$ is $q$-coercive. Then, by \eqref{rate1} and \eqref{rate2} we have the estimates
$$
\|\xd - \xdag\| \leq \left[ \displaystyle\frac{1}{\zeta}\displaystyle\frac{1+\beta_1}{1-\beta_1}D_{\xi^\dagger}(\overline{x_m},\xdag) 
+ \frac{1}{\zeta}\frac{\beta_2}{1-\beta_1}(\tau+\lambda+2)\delta \right]^{\frac{1}{q}}
$$
and
$$
\|\xd - \xdag\| \leq \left[\displaystyle\frac{1}{\zeta}\frac{\beta_2(1+\lambda)}{1-\beta_1}\delta \right]^{\frac{1}{q}},
$$
respectively.
\end{rem}

\section{Discrete Forward Operator}\label{sec:addendum}
We now present some aspects to be considered when we replace the continuous forward operator by a finite-dimensional approximation.

Let us consider a sequence of finite-dimensional subspaces $\{Y_n\}_{n\in \N}$ of the space $Y$, such that
$$
Y_n \subset Y_{n+1} \subset ... \subset Y ~~~\text{ and }~~~ \overline{\cup_{n\in\N}Y_n} = Y.
$$
Then, we replace the continuous forward operator by some finite-dimensional approximation. In Section~\ref{sec:numerics}, we consider as an illustrative example, the discretization of the parameter to solution map that associates a diffusion parameter to the solution of a Parabolic Cauchy problem. The discretization is then defined by Crank-Nicolson scheme that solves numerically the associated parabolic partial differential equation.

In the present discrete setting, we consider the following alternative discrepancy principle:
\begin{df}
Let $\delta> 0$ and $\yd$ be fixed. For $\lambda >\tau > 1$, we choose $m,n \in \N$ and $\alpha > 0$, with $m = m(\delta,\yd)$, $n = n(\delta,\yd)$ and $\alpha = \alpha(\delta,\yd)$, such that
\begin{equation}
 \tau \delta \leq \|F_n(x^{\delta,\alpha}_{m,n}) - \yd\| \leq \lambda\delta,
\label{discm3}
\end{equation}
holds for $x^{\delta,\alpha}_{m,n}$, a solution of
\begin{equation}
  \min\{\|F_n(x) - \yd\|^p + \alpha f_{x_0}(x)\} ~~~\text{ subject to }~~~ x \in \mathcal{D}_m.
\label{tik4}
\end{equation}
\end{df}

In the present context, all the results of the previous sections hold. However, some additional calculations should be done when $F$ is replaced by $F_m$. The main argument in the convergence analysis is based on the  existence of a diagonal subsequence converging (weakly) to an $f_{x_0}$-minimizing solution of Problem~\ref{prob1}, when the limits $\delta \rightarrow 0$, $m,n\rightarrow \infty$ are taken.

More precisely, when $\delta > 0$ is fixed, the limit $n \rightarrow \infty$ is taken and the discrepancy principle \eqref{discm3} holds true for every $n$. Then, we can find a sequence of minimizers $\{x^{\delta,\alpha}_{m,n}\}_{n,\in\N}$, converging weakly to some minimizer of (\ref{tik2}), satisfying \eqref{discm2}. By Proposition~\ref{pr:conv}, if we also take the limit $m \rightarrow \infty$, the resulting sequence has a weakly convergent subsequence with limit satisfying the continuous version of the Morozov discrepancy principle presented in \cite{anram}. For this reason, we can always assume the existence of a diagonal subsequence converging (weakly) to an $f_{x_0}$-minimizing solution of Problem~\ref{prob1} when $\delta \rightarrow 0$.

The proof of these results in the specific example of local volatility calibration by Tikhonov regularization can be found in Section 4 of \cite{acz2013a}.

\section{An Alternative Discrepancy Principle}\label{sec:addendumb}

In general, Assumption~\ref{ass4} does not hold for nonlinear forward operators. See \cite[Remark~4.7]{schu}. More precisely, one of the inequalities of the discrepancy principle~\eqref{discm2} is not satisfied with prescribed constants $1 < \tau \leq \lambda$ or $1 < \tau_1 \leq \tau_2$. Thus, as an alternative, whenever  ensuring \eqref{discm2} is not possible, we base our choice of $\alpha$, for a fixed $m$, on the sequential discrepancy principle, presented in \cite{ahm}. It goes as follows:

\begin{df}[Sequential Morozov Criteria]
For prescribed $\tilde{\tau}>1$, $\alpha_0 > 0$ and $0<q<1$, we choose $k \in \N$ such that $\alpha_k := q^k\alpha_0$ satisfies the discrepancy
\begin{equation}
\|F(x^\delta_{m,\alpha_{k}}) - y^\delta \| \leq \tilde{\tau}\delta < \|F(x^\delta_{m,\alpha_{k-1}}) - y^\delta \|,
\label{seqmorozov}
\end{equation}
for some $x^\delta_{m,\alpha_{k}} \in M_{\alpha_k,m}$ and $x^\delta_{m,\alpha_{k-1}} \in M_{\alpha_{k-1},m}$.
\end{df}

The existence of $\alpha$ and $m$ satisfying the discrepancy \eqref{seqmorozov}, follows directly from Proposition~\ref{pr7} and by assuming that $m$ is sufficiently large. More precisely, we can replace, for instance, $\tilde{\tau}\delta$ in \eqref{seqmorozov} by $(1+\epsilon)(\gamma_m+\delta)$. Then, the estimate
$$
\|F(x^\delta_{m,\alpha_{n}}) - y^\delta \| \leq (\tilde{\tau}-\epsilon)(\gamma_m+\delta) < \|F(x^\delta_{m,\alpha_{n-1}}) - y^\delta \|
$$
always hold by Proposition~\ref{pr7}, for every fixed $m$ and $\epsilon \in (0,\tilde{\tau}-1)$. Thus, for a sufficiently large $m$, it follows that $\tilde{\tau}\delta \approx (\tilde{\tau}-\epsilon)(\gamma_m+\delta)$.

Theorems~\ref{prop:conv1} and \ref{theo:conv} remain valid if the discrepancy principle~\ref{discm2} is replaced by the sequential discrepancy principle \eqref{seqmorozov}. This follows by noting that, whenever the lower inequality in the discrepancy principle \eqref{discm2} holds, we can replace it by $\tilde{\tau} \leq \|F(x^\delta_{m,\alpha/q}) - y^\delta\|$, where $\alpha$ satisfies the lower inequality of the sequential discrepancy principle \eqref{seqmorozov} and $\alpha/q$ satisfies the upper one. See \cite[Section~3]{ahm}.

The discrepancy principle~\eqref{discm2} is always preferable, since its lower inequality implies that $\|F(x^\delta_{m,\alpha_{n}}) - y^\delta \| \geq \tau\delta$. This avoids the Tikhonov solutions to over fit and to reproduce noise. On the other hand, the same conclusion is not necessarily true if the sequential discrepancy principle~\eqref{seqmorozov} is used.

When using the sequential discrepancy principle~\eqref{seqmorozov}, it is not alway possible to achieve the rate of convergence $D_{\xi^\dagger}(x^{\delta}_{m,\alpha},ax^\dagger) = \mathcal{O}(\delta)$. The technical point is that, if $\alpha$ satisfies the lower inequality of \eqref{seqmorozov}, then the estimate $f_{x_0}(x^{\delta}_{m,\alpha}) - f_{x_0}(x^\dagger) \leq 0$ does not necessarily holds. Such estimate holds for $\alpha/q$, instead. An additional condition for achieving the convergence rate $D_{\xi^\dagger}(x^{\delta}_{m,\alpha},ax^\dagger) = \mathcal{O}(\delta)$ with the sequential discrepancy principle~\eqref{seqmorozov} is to assume that $\alpha = \mathcal{O}(\delta)$. For a more detailed discussion about convergence rates under the sequential discrepancy principle \eqref{seqmorozov}, see \cite[Section~4]{ahm} and \cite{hoffmathe}.

\section{Numerical Examples}\label{sec:numerics}

We shall now illustrate the theoretical results of the previous sections with some numerical examples based on the calibration of a diffusion coefficient in a parabolic problem. See \cite{acz2013a,crepey,acpaper,eggeng}. More precisely, let $a_1,a_2 \in \mathbb{R}$ be scalar constants such that $0 < a_1 \leq a_2 < +\infty$ and let $a_0 \in \he$ be fixed. Define the set
$$
Q := \{a \in a_0 + \he : a_1\leq a \leq a_2\}.
$$
Assuming that the data $u$ was generated by the following parabolic problem:
\begin{equation}
\left\{
\begin{array}{rcll}
\displaystyle\frac{\partial u}{\partial \tau} - 
a(\tau,y)\left(\frac{\partial^2 u}{\partial y^2} - \frac{\partial u}{\partial y}\right) - 
b\frac{\partial u}{\partial y} &=& 0 & \tau > 0, ~y \in \mathbb{R}\\
\\
u(\tau = 0,K) &=& \max\{0,1-\text{e}^y\}, & \text{for}~ y \in \mathbb{R},\\
\\
\displaystyle\lim_{y\rightarrow +\infty}u(\tau,y) & = & 0,&\text{for }~ \tau > 0,\\
\\
\displaystyle\lim_{y\rightarrow 0^-\infty}u(\tau,y) & = & 1,&\text{for }~ \tau > 0,
\end{array}
\right.
\label{dup1}
\end{equation}
\noindent our problem is to find the diffusion parameter $a\in Q$.

\noindent We define the forward operator by:
$$
\begin{array}{rcl}
 F: Q \subset\he&\longrightarrow & L^2(\mathbb{R}_+\times\mathbb{R})\\
   a &\longmapsto & u(a) - u(a_0),
\end{array}
$$
with $a_0 \in Q$ fixed and {\it a priori} chosen. The choice of $\he$ is justified in \cite{acpaper,eggeng}.

The forward operator under consideration fulfills the hypotheses of the previous sections theorems. Here, we shall implement numerically the Tikhonov regularization for this specific problem with synthetic data. For the technical details, see \cite{acz2013a,crepey,acpaper,eggeng}.

In the calibration we take as true (known) diffusion coefficient the following:
\begin{equation}
\sigma(\tau,y) = \left\{
\begin{array}{ll}
\displaystyle\frac{2}{5}-\frac{4}{25}\text{e}^{-\tau/2}\cos\left(\displaystyle\frac{4\pi y}{5}\right),& \text{ if } -2/5 \leq y \leq 2/5\\
\\
2/5,& \text{ otherwise,}
\end{array} \right.
\label{vol}
\end{equation}
and set $a = \sigma^2/2$.
We also assume that $b = 0.03$ in Equation~\eqref{dup1}.

We illustrate the discrepancy principle and the convergence-rate results of the previous sections by changing the noise and discretization levels.

The forward problem defined in Equation~\eqref{dup1} as well as its adjoint, arising in the evaluation of the gradient of the Tikhonov functional are numerically solved in the domain $D = [0,1]\times[-5,5]$ by the Crank-Nicolson scheme, 
\begin{multline}
u^{n+1}_{m} - \displaystyle\frac{1}{2}\eta a^{n+1}_{m}(u^{n+1}_{m+1} - 2u^{n+1}_{m} +u^{n+1}_{m-1}) + \frac{1}{4}\alpha (a^{n+1}_{m}-b)(u^{n+1}_{m+1} - u^{n+1}_{m-1})\\ =
u^{n}_{m} + \displaystyle\frac{1}{2}\eta a^{n}_{m}(u^{n}_{m+1} - 2u^{n}_{m} +u^{n}_{m-1}) - \frac{1}{4}\alpha (a^{n}_{m}-b)(u^{n}_{m+1} - u^{n}_{m-1}).
\label{cns}
\end{multline}
with the boundary conditions:
$$
\displaystyle\lim_{y\rightarrow-5}u(\tau,y) = 1 ~~\text{ and }~~ 
\displaystyle\lim_{y\rightarrow 5}u(\tau,y) = 0
$$
See \cite{acz2013a,vvla2,acpaper}.

We generate the data as follows:\\
On a given mesh, we numerically solve the Cauchy problem of Equation~\eqref{dup1} with the diffusion coefficient given in Equation~\eqref{vol}, which is evaluated on the same mesh. We add a zero-mean Gaussian noise with standard deviation $0.01$ to this numerical solution and interpolate the resulting data in a coarser mesh. We then use this data to calibrate the corresponding diffusion coefficient.

If $\Delta \tau$ and $\Delta y$ denote the time and space mesh sizes, respectively, the calibration of the diffusion coefficient is numerically solved with different values of $\Delta \tau$, $\Delta y$ and the regularization parameter $\alpha$, until the discrepancy principle
\begin{equation}
\tau \delta \leq \|u(a^{\delta}_{m,\alpha}) - u^\delta\| \leq \lambda \delta,
\label{discrepancy_num} 
\end{equation}
is satisfied, where $u^\delta$ is the noisy data and $\delta >0$ is the noise level. Note that, by the definition of $F$, $u(a) - u^\delta = F(a) - (u^\delta - u(a_0) )$. Thus, instead of using $F(a^{\delta}_{m,\alpha})$ in the discrepancy principle of Equation~\eqref{discrepancy_num}, we simply use $u(a^{\delta}_{m,\alpha})$, with no loss of generality.

We use this data to calibrate the diffusion coefficient by Tikhonov regularization with the smoothing penalization:
$$
f_{a_0}(a) = \beta_1\|a-a_0\|^2_{L_2(D)} + \beta_2\|\partial_y (a-a_0)\|^2 + \beta_3\|\partial_\tau (a-a_0)\|^2,
$$
with $\beta_1 = 0.5$, $\beta_2 = 0.25 \Delta y$, $\beta_3 = 0.25 \Delta \tau$ and $a_0 \equiv 0.08$.

The minimization of the Tikhonov functional is performed recursively by the gradient method. More precisely, if $a^{k}$ denotes the diffusion coefficient at the $k$th iteration, the next step is given by
$$
a^{k+1} = a^{k} - \lambda_k \nabla \mathcal{F}^{u^\delta}_{\alpha,a_0}(a^{k}),
$$
until the Morozov discrepancy principle is satisfied or the maximum number of iterations is reached or yet the relative change in the residual is less than $1.0\times 10^{-4}$. We base the choice of the step-length $\lambda_k$ on the Wolfe rule. The algorithm is initialized with the step length $\lambda_k^{0} = \|\mathcal{F}^{u^\delta}_{\alpha,a_0}(a^{k-1})\|^2/\|\mathcal{F}^{u^\delta}_{\alpha,a_0}(a^{k})\|^2$. See Algorithm 3.5 and Algorithm 3.6 in Chapter III of \cite{nocedal}. The parameters used in the Wolfe conditions are $c_1 = 10^{-8}$ and $c_2 = 0.95$. The iterations begun with $a^0 = a_0 \equiv 0.08$.

The data is generated with step sizes $\Delta \tau = 0.0025$ and $\Delta y = 0.01$ and the coarser grid is given by the step lengths $\Delta \tau = 0.02$ and $\Delta y = 0.1$. 
In the numerical solution of the inverse problem, Equation~\eqref{dup1} is numerically solved in the same mesh we interpolate the data, i.e., we use $\Delta \tau = 0.02$ and $\Delta y = 0.1$ in both cases. We vary the mesh used to evaluate the diffusion coefficient in order to highlight the discrepancy principle \eqref{discrepancy_num}. The step sizes used in the tests were the following:
$$
\Delta \tau = 0.1, 0.08, 0.07, 0.06, 0.05, 0.04, 0.03, 0.02, 0.01, 0.0075, 0.005, 0.0025
$$
and
$$
\Delta y = 0.25, 0.22, 0.20, 0.17, 0.15, 0.13, 0.11 ,0.1, 0.05, 0.04, 0.02, 0.01.
$$
Figures~\ref{test1} and \ref{test2} presents the residual and the error estimates associated to the regularized solutions for the above meshes, respectively.

We stress that, in the present set of examples, we chose $\tau = 1.025$ and $\lambda = 1.125$ in the discrepancy principle \eqref{discrepancy_num} and the illustration of this discrepancy principle can be found in Figure~\ref{test1}.

If $\tilde{u}$ and $P_{n_0}(\tilde{u} + e)$ denote the full noiseless data and the noisy data in the coarser grid with noise $e$, respectively, then the noise level can be estimated by:
$$
\delta = \|\tilde{u} - P_{n_0}(\tilde{u} + e)\|_{L^2(D)} = 
\left(\displaystyle\int_0^1\int_{-5}^{5}|\tilde{u}(\tau,y) - P_{n_0}(\tilde{u} + e)(\tau,y)|^2dy d\tau  \right)^{1/2}.
$$
This integral is solved by the 2D-Simpson's rule and the data was interpolated linearly.

\begin{figure}[ht]
\centering
\includegraphics[width=0.5\textwidth]{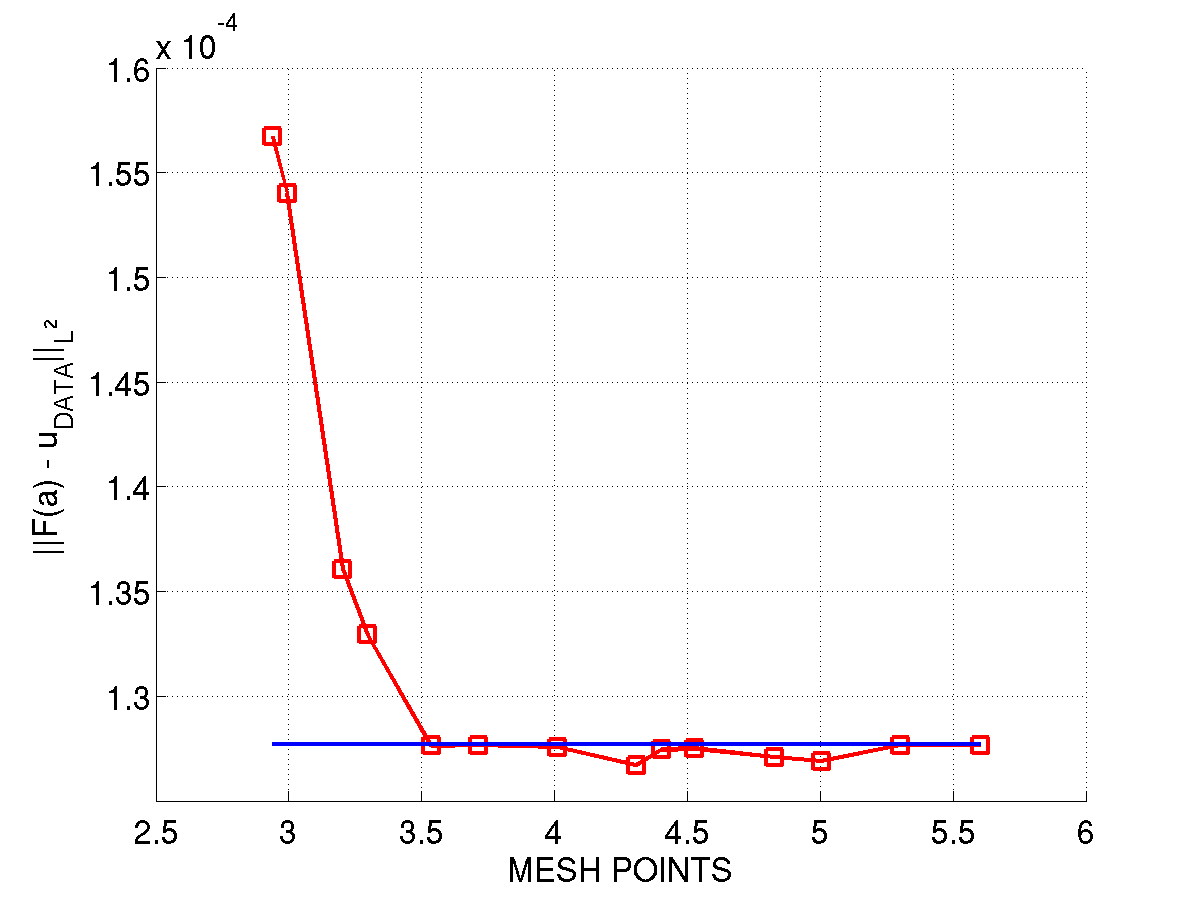}
\caption{Evolution of the residual as a function of the number of mesh points. We choose the regularization parameter presenting lower residual. In the presence of noise, some discretization levels in the domain satisfy the discrepancy principle. Compare it to the error estimation in Figure~\ref{test2}. The horizontal line corresponds to $\lambda\delta$.}
\label{test1}
\noindent 
\end{figure}

\begin{figure}[ht]
\centering
\includegraphics[width=0.5\textwidth]{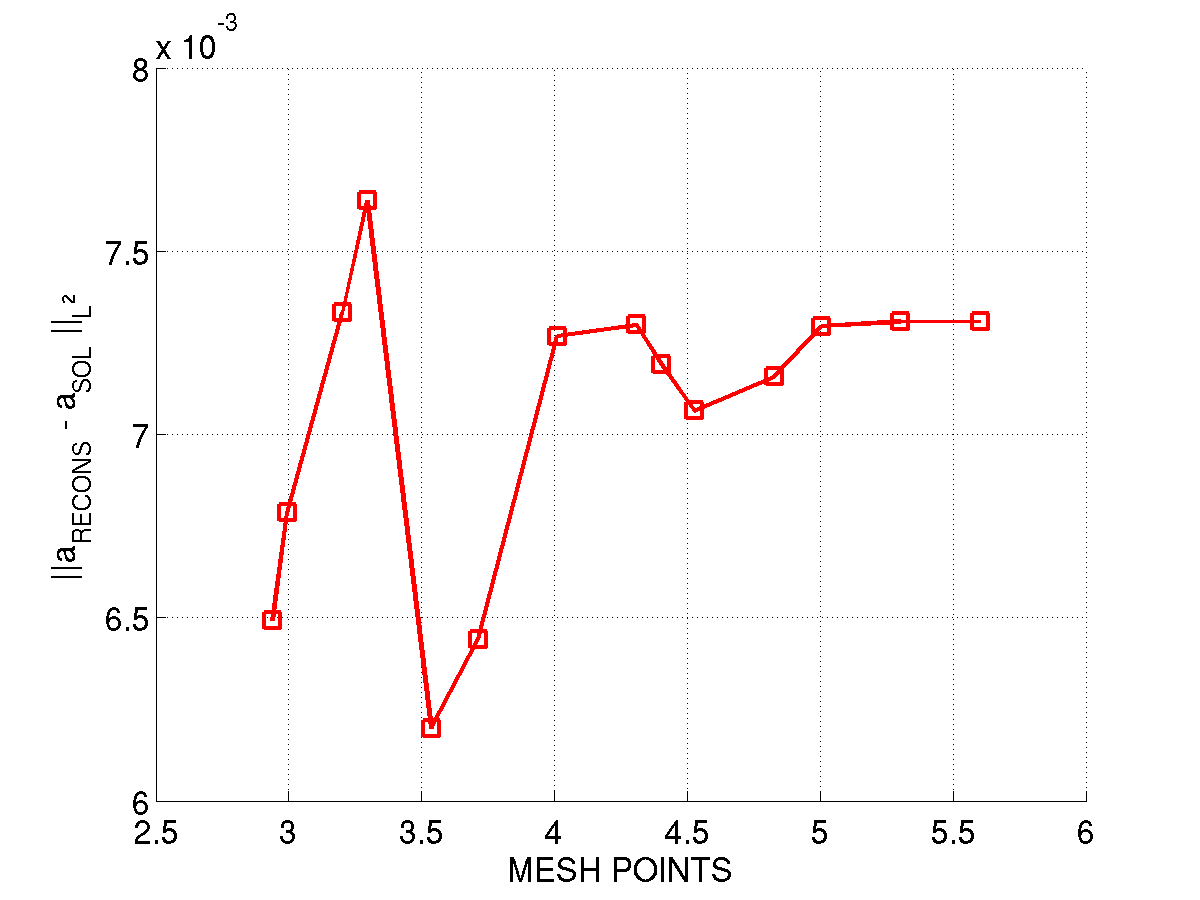}
\caption{Evolution of the $L^2$-error. In the presence of noise, its minimum is attained for a coarser mesh satisfying the discrepancy principle of Equation~\eqref{discrepancy_num}.}
\label{test2}
\noindent 
\end{figure}

\begin{figure}[ht]
\centering
\includegraphics[width=0.33\textwidth]{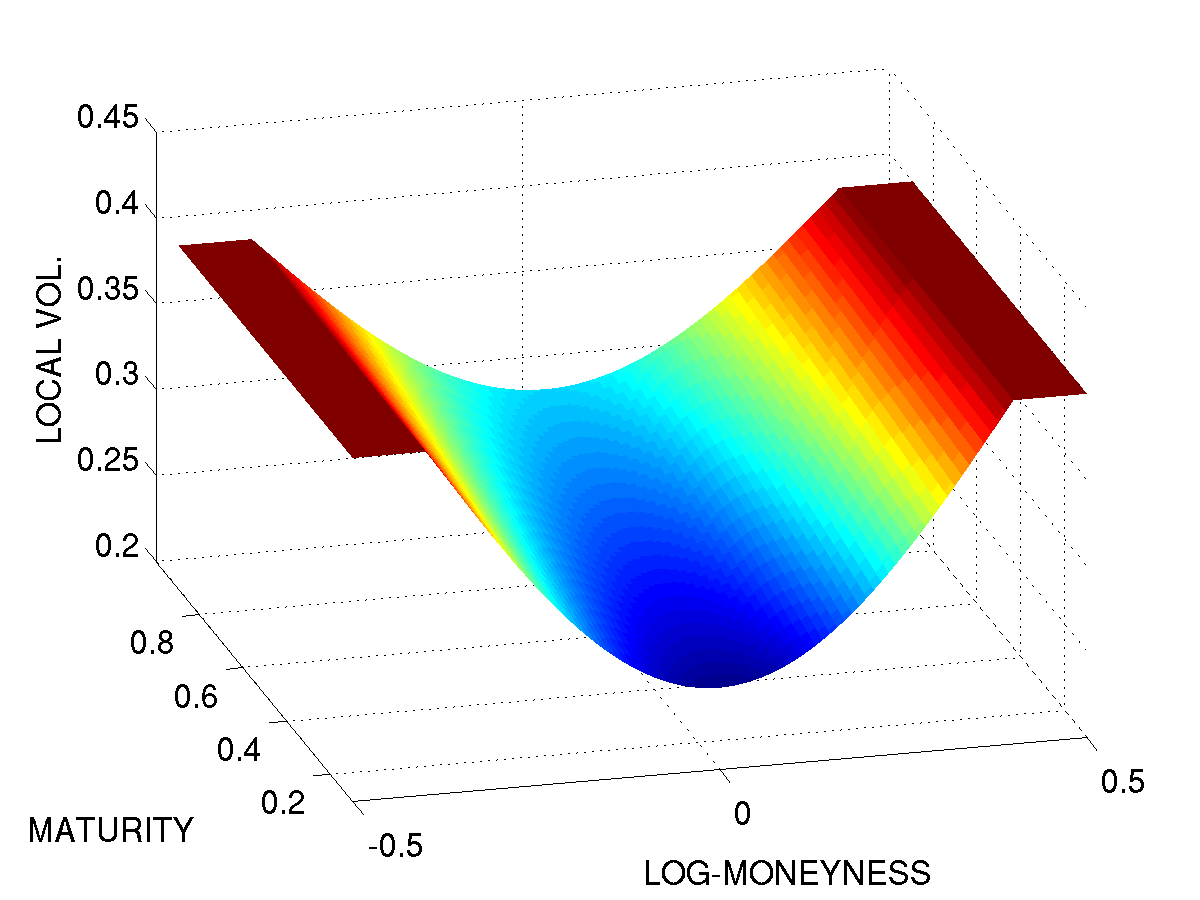}\hfill 
\includegraphics[width=0.33\textwidth]{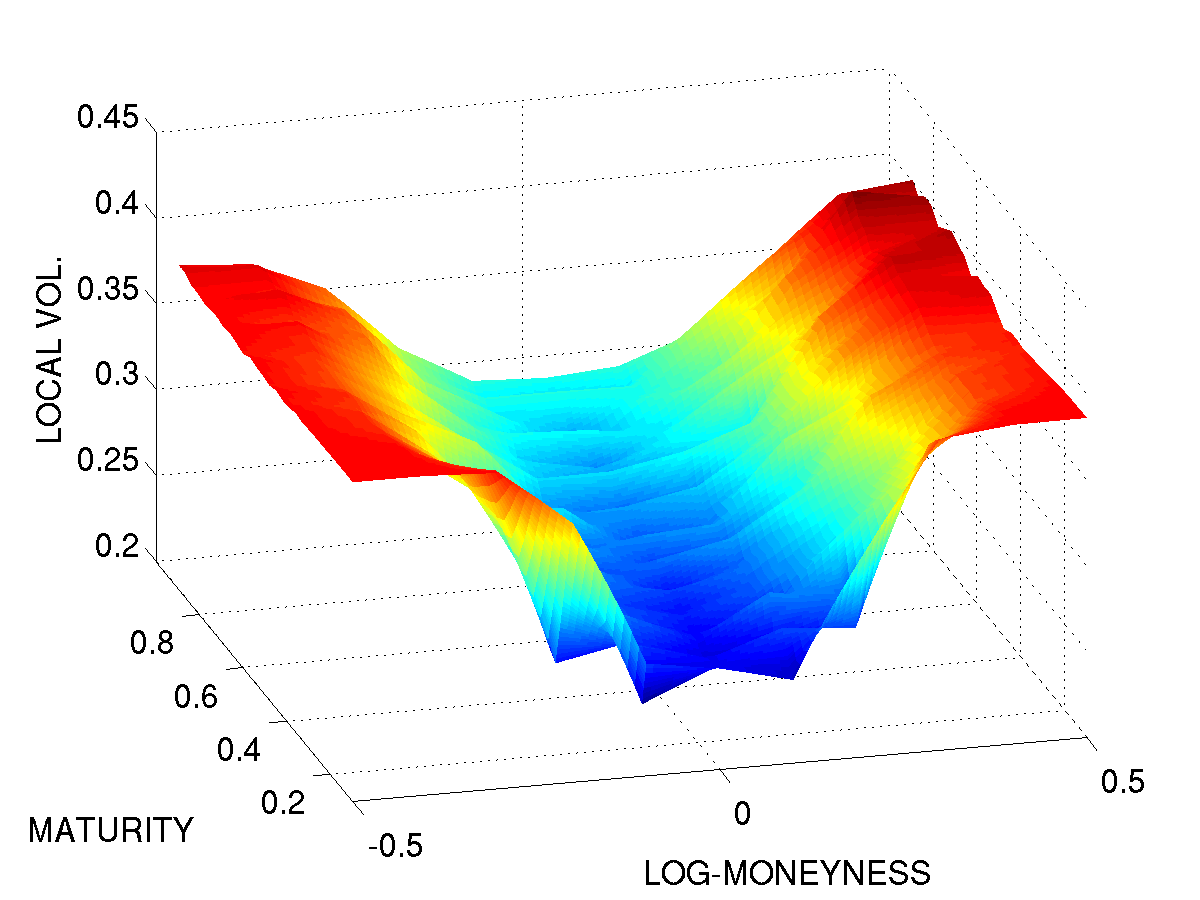}\hfill
\includegraphics[width=0.33\textwidth]{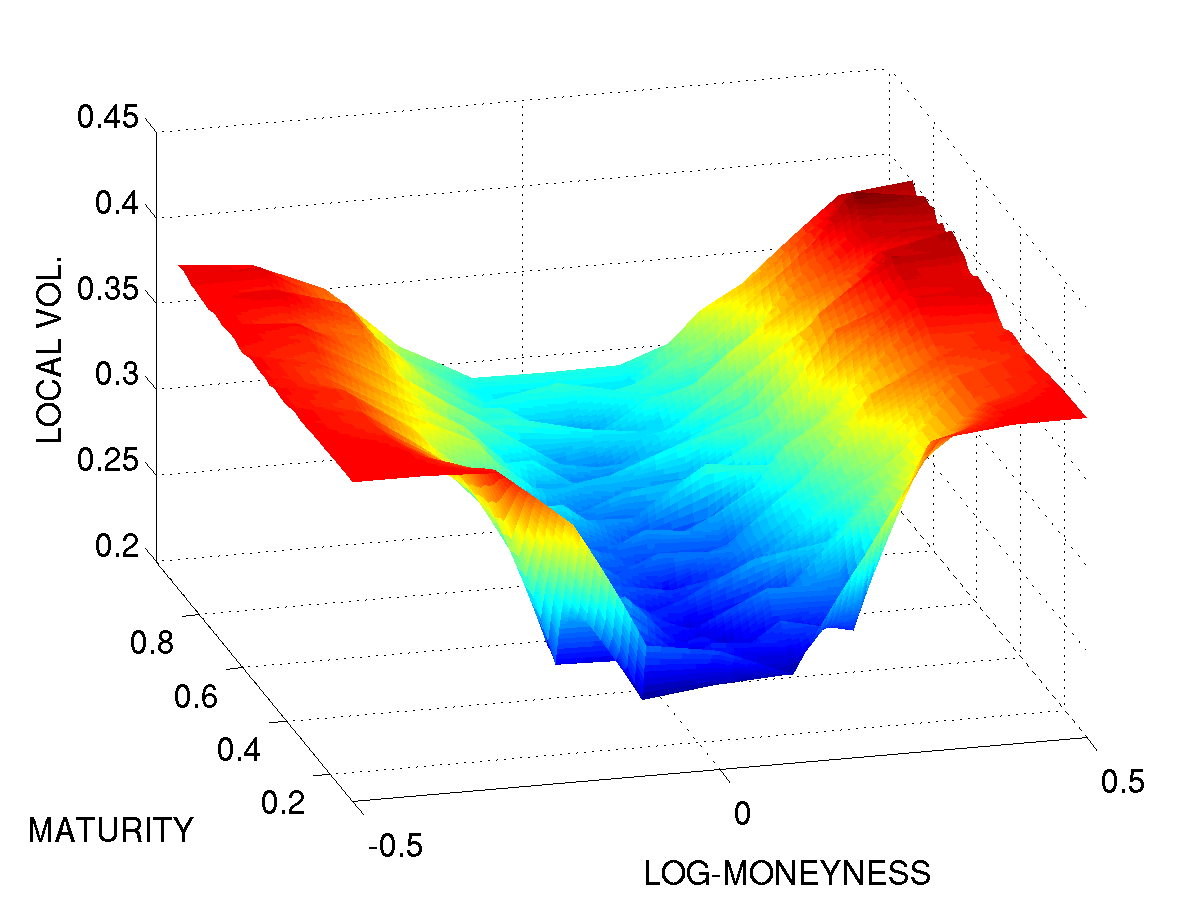}
\caption{Left: original surface. Center and right: reconstructions corresponding to the first and second points satisfying the discrepancy principle of Figure~\ref{test1}, respectively.}
\label{test3}
\noindent 
\end{figure}

\begin{figure}[ht]
\centering
\includegraphics[width=0.33\textwidth]{original}\hfill 
\includegraphics[width=0.33\textwidth]{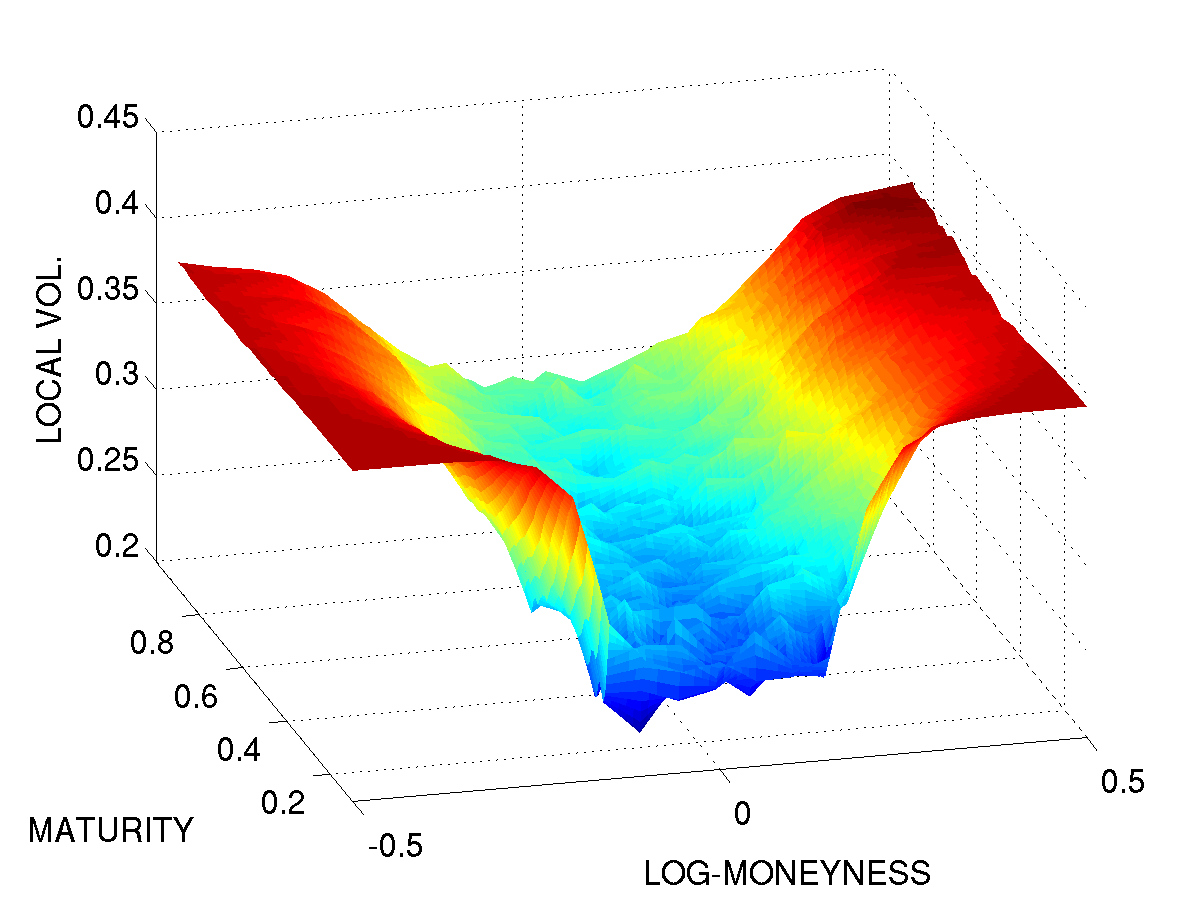}\hfill
\includegraphics[width=0.33\textwidth]{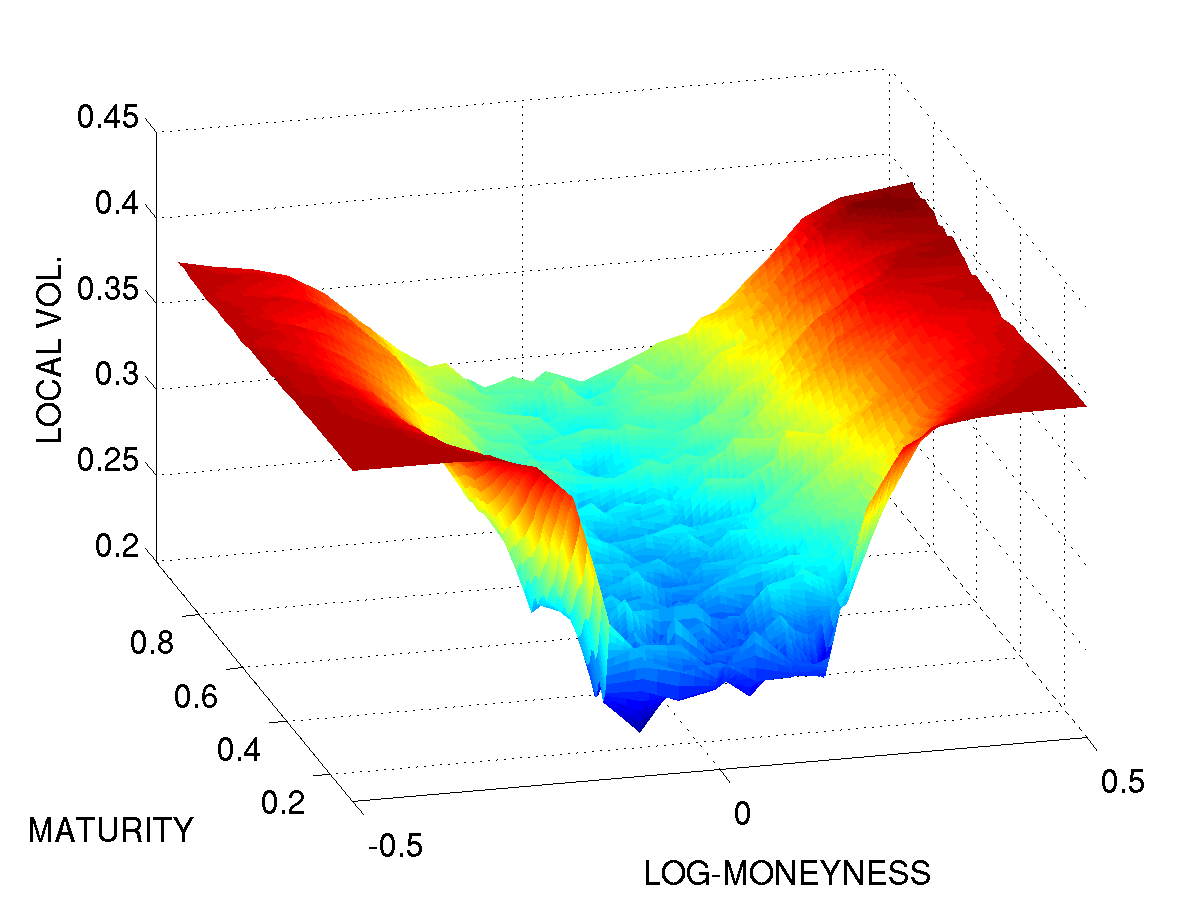}
\caption{Left: original surface. Center and right: reconstructions satisfying the discrepancy principle of Figure~\ref{test1}.}
\label{test4}
\noindent 
\end{figure}
The values of the regularization parameter used in the present test were:
$$
\beta = 0.25,~ 0.10,~ \sqrt{\delta},~ 0.01,~ 0.006,~ \delta,~ 0.001,~ 5.0 \times 10{-4},~ 1.0\times 10^{-4},~ 5 \times 10^{-5},~ \delta^2,~ 0.
$$
In Figures~\ref{test1}, \ref{test2}, \ref{test3} e \ref{test4}, we have chosen the reconstructions with regularization parameter presenting the lowest residual satisfying the discrepancy principle of Equation~\eqref{discrepancy_num}.

We also calculate the $L^2$-error, i.e., the $L^2(D)$ distance between the regularized solution and the original diffusion coefficient. The resulting $L^2$-error for the regularized solutions used in Figure~\ref{test1} can be found in Figure~\ref{test2}. Note that, reconstructions with coarser meshes satisfying the discrepancy principle of Equation~\eqref{discrepancy_num} presented satisfactory $L^2$-error estimates, illustrating the reliability of its use for finding the appropriate discretization level in the domain and the regularization parameter.

Figures~\ref{test3} and \ref{test4} present reconstructions satisfying the discrepancy principle of Equation~\eqref{discrepancy_num}. Note that, the reconstructions with coarser grid satisfying the discrepancy in Figure~\ref{test1} presented better $L^2$-error estimates. Moreover, the surfaces displayed in Figure~\ref{test3} are smoother than those of Figure~\ref{test4}.

\section{Conclusions}
Finding appropriate discretization levels is a well known challenge when solving Tikhonov-type regularization problems. In this work, we have shown that the Morozov discrepancy principle could also be used to find it appropriately. Since we are working in a discrete setting, some additional assumptions ought be made in order to establish theoretical results. 

Under the above mentioned discrepancy-based choices, we also presented a convergence analysis with convergence rates in terms of the noise level. In addition, we presented some guidelines on how to apply these results when the forward operator is replaced by a discrete approximation. We also apply the sequential discrepancy principle given by Equation~\eqref{seqmorozov}, for this discrete setting.

A numerical example illustrated the discrepancy principle when the noise level and the  discretization level of the forward operator are kept fixed and the discretization level in the domain is varied.

Summing up, Morozov principle is a robust rule for determining appropriately the regularization parameter and discretization levels in Tikhonov regularization.

\addcontentsline{toc}{section}{Acknowledgements}
\section*{Acknowledgements}
V.A. and J.P.Z. acknowledges and thanks CNPq, Petroleo Brasileiro S.A. and Ag\^encia Nacional do Petr\'oleo for the financial support during the time when this work was developed. 
A.D.C. acknowledges and thanks the financial support from CNPq through grant 200815/2012-1 and FAPERGS through grant 0839 12-3.
 J.P.Z. also acknowledges and thanks the financial support from CNPq through grants 302161/2003-1 and
474085/2003-1, and from FAPERJ through the programs {\em Cientistas do Nosso Estado}.


\appendix
\section{Regularizing Properties of the Discrete Morozov's Principle}\label{app:morozov}

\begin{theo}[Convergence]
The regularizing parameter $\alpha = \alpha(\delta,\yd,\gamma_m)$ obtained through the discrepancy principle of Definition~\ref{def6} satisfies the limits
\begin{equation}
\displaystyle\lim_{\delta,\gamma_m \rightarrow 0+}\alpha(\delta,\yd,\gamma_m) = 0
~~~\text{       and       }~~~
\displaystyle\lim_{\delta,\gamma_m \rightarrow 0+}\frac{(\delta + \gamma_m)^p}{\alpha(\delta,\yd,\gamma_m)} = 0.
\label{eq:estimates1}
\end{equation}
\label{tma}
\end{theo}

\begin{proof} 
Consider the sequences $\{\delta_k\}_{k\in\N}$ and $\{\gamma_{m_k}\}_{k\in\N}$ converging monotonically to zero. Define the sequence $\{\alpha_k\}_{k\in\N}$ by setting $\alpha_k$ to be the regularization parameter $\alpha(\delta_k,\gamma_{m_k})$ satisfying Definition~\ref{def6} for each $k$. Thus, for each $\alpha_k$ we can find $x_k = x^{\delta_k}_{m_k,\alpha_k}$, a solution of Problem~\ref{tik2}. We thus define the sequence $\{x_k\}_{k\in\N}$. By the pre-compactness of the level sets of the Tikhonov functional, it follows that $\{x_k\}_{k\in\N}$ has a weakly convergent subsequence, denoted by $\{x_l\}_{l\in\N}$, with weak limit $\tx \in \domainf$.

\noindent By the weak lower semi-continuity of the norm and the weak continuity of $F$, the following estimates hold:
$$
\|F(\tx)-y\|\leq\displaystyle\liminf_{l\rightarrow\infty}\|F(x_l)-y^{\delta_{l}}\| + \delta_l \leq \lim_{l\rightarrow\infty}(\tau_2+1)(\gamma_{m_l}+\delta_l) = 0.
$$
Note that in the above estimates we have used $l$ instead of $k_l$ to easy notation. Note also that, $\tx$ is a least-square solution of Problem~\ref{prob1}. 

\noindent We also have the estimates:
$$
\begin{array}{rcl}
 \tau_1(\delta_l + \gamma_{m_l})^p + \alpha_l f_{x_0}(x_l) &\leq& 
(\delta_l + \gamma_{m_l})^p + \alpha_l f_{x_0}(P_{m_l}\xdag).
\end{array}
$$
\noindent Since $\tau_1 > 1$, it follows that $f_{x_0}(x_l) \leq f_{x_0}(P_{m_l}\xdag)$, which implies that $f_{x_0}(\tx) \leq f_{x_0}(\xdag)$.
Hence, $\tx$ is an $f_{x_0}$-minimizing solution of Problem~\ref{prob1}.

\noindent Assume that there exists $\overline{\alpha}>0$ and a subsequence $\{\alpha_{l_n}\}_{n\in\N}$ such that $\alpha_{l_n} \geq \overline{\alpha}$. Then, take the respective subsequence of minimizers $\{x_{l_n}\}_{n\in\N}$.

\noindent Define the sequence of minimizers $\{\overline{x_n}\}_{n\in\N}$, with $\overline{x_n} := x^{\delta_{l_n}}_{m_{l_n},\alpha_{l_n}}$. Since $L$ is non-decreasing, it follows that
$$
\|F(\overline{x_n})-y^{\delta_{l_n}}\| \leq \|F(x_{l_n}) - y^{\delta_{l_n}}\|\leq \tau_2(\delta + \gamma_m) \rightarrow 0.
$$
Note that, since $\xdag,\tx$ are $f_{x_0}$-minimizing solutions of Problem~\ref{prob1}, it follows that $f_{x_0}(\xdag) = f_{x_0}(\tx)$.

On the other hand, 
$$
\begin{array}{rcl}
\displaystyle\limsup_{n\rightarrow\infty}\overline{\alpha}f_{x_0}(\overline{x_n}) &\leq&
\overline{\alpha} f_{x_0}(\xdag).
\end{array}
$$
By the weak pre-compactness of the level sets $\mathcal{M}_{\alpha}(\rho)$, 
it follows that  $\{\overline{x_n}\}_{n\in\N}$ has a weakly convergent subsequence with limit $\overline{x}$. By the above estimates, it follows that $\overline{x}$ is an $f_{x_0}$-minimizing solution for Problem~\ref{tik2}. Denoting this subsequence by $\{\overline{x_n}\}_{n\in\N}$, it follows by the above estimates that
$$
\begin{array}{rcl}
 \|F(\overline{x}) - y\|^p + \overline{\alpha} f_{x_0}(\overline{x}) &\leq& \displaystyle\liminf_{n\rightarrow\infty}\left( \|F(\overline{x_n})-y^{\delta_{l_n}}\|^p + \overline{\alpha} f_{x_0}(\overline{x_n}) \right)\\
&\leq& \|F(x) - y\|^p + \overline{\alpha} f_{x_0}(x) \text{ for every } x \in X.
\end{array}
$$
Then, $\overline{x}$ is a solution of Problem~\ref{tik2} with data $y$ (noiseless) and regularization parameter $\overline{\alpha}$.

\noindent Since $f_{x_0}$ is convex and $f_{x_0}(x) = 0$ if, and only if, $x = x_0$, it follows that, for every $t \in [0,1)$ we have
$$
f_{x_0}((1-t)\overline{x} + t x_0) \leq (1-t)f_{x_0}(\overline{x}) + t f_{x_0}(x_0) = (1-t)f_{x_0}(\overline{x}).
$$
It also follows that
$$
\begin{array}{rcl}
 \overline{\alpha}f_{x_0}(\overline{x}) 
 &\leq & \|F((1-t)\overline{x} + t x_0) - y\|^p + \overline{\alpha} (1-t)f_{x_0}(\overline{x}).
\end{array}
$$
This shows that,
$$
\overline{\alpha}tf_{x_0}(\overline{x})\leq \|F((1-t)\overline{x} + t x_0) - y\|^p.
$$
Then, Assumption~\ref{ass3} implies that $f_{x_0}(\overline{x}) = 0$, and thus $\overline{x} = x_0$. This is a contradiction since
$$
\|F(x_0)-y\| \geq \|F(x_0)-\yd\|-\|\yd-y\|\geq(\tau_1 - 1)\delta>0.
$$
Therefore, the first limit in \eqref{eq:estimates1} holds.

In order to prove the second limit we proceed as follows. Since $\{x_l\}_{l\in\N}$ weakly converges to $\tx$ with $f_{x_0}(x_l)\rightarrow f_{x_0}(\tx)$, it follows that
$$
\begin{array}{rcl}
 \tau_1^p(\delta_l + \gamma_{m_l})^p + \alpha_l f_{x_0} (x_l) 
 &\leq & (\gamma_m + \delta_l)^p + \alpha_l f_{x_0}(P_{m_l}\xdag)
\end{array}
$$
The above estimate combined with the limit $f_{x_0}(P_{m_l}\xdag)\rightarrow f_{x_0}(\xdag)$ leads to
$$
(\tau_1^p-1)\displaystyle\frac{(\delta_l+\gamma_{m_l})^p}{\alpha_l} \leq f_{x_0}(P_{m_l}\xdag) - f_{x_0}(x_l)
$$
where the right hand side converges to zero when $l \rightarrow \infty$. Note that, we have used again the fact that $f_{x_0}(\xdag)=f_{x_0}(\tx)$ if $\xdag,\tx \in \mathcal{L}$.

\end{proof}

Choosing $m$ based on Equation~\eqref{def4}, we have the following corollary:

\begin{cor}[Convergence]
Let $\alpha$ satisfy Definition~\ref{def6}. Then,
\begin{equation}
 \displaystyle\lim_{\delta\rightarrow 0}\alpha(\delta,\gamma_m(\delta)) = 0 ~~\text{ and }~~ 
\displaystyle\lim_{\delta\rightarrow 0}\frac{\delta^p}{\alpha(\delta,\gamma_m(\delta))} = 0.
\label{est1a}
\end{equation}
Moreover, we have the convergence result of Theorem~\ref{tma} with this choice of $m$.
\label{conv}
\end{cor}
\begin{proof} By Theorem~\ref{tma}, $\alpha = \alpha(\delta,\yd,\gamma_m)$ satisfies the limits:
\begin{equation}
 \displaystyle\lim_{\delta,\gamma_m\rightarrow 0}\alpha(\delta,\yd,\gamma_m) = 0 ~~\text{ and }~~ 
\displaystyle\lim_{\delta,\gamma_m\rightarrow 0}\frac{(\delta+\gamma_m)^p}{\alpha(\delta,\yd,\gamma_m)} = 0.
\label{est1}
\end{equation}
Then, following the same arguments in the proof of Equation~\eqref{eq:estimates1} in Theorem~\ref{tma} and substituting $\tau_1(\delta + \gamma_m)$ by $\tau_1\delta$ and dominating $\delta + \gamma_m$ by $\lambda/\tau_2 \delta$ based on \eqref{def4}, it follows that the limits in Equation~\eqref{est1a} hold.

Consider the sequence of positive constants $\{\delta_{k}\}_{k\in\N}$ converging monotonically to zero and define the sequence $\{m_k\}_{k\in\N}$, with $m_k := m(\delta_k,y^{\delta_k})$ satisfying \eqref{def4}. Thus, we can choose a sequence $\{x_k\}_{k\in\N}$ of solutions of Problem~\ref{tik2} with $x_k := x^{\delta_k}_{\alpha_k,m_k}$ and $\alpha_k$ satisfying Definition~\ref{def6}.
Then, the convergence of a subsequence, denoted by $\{x_l\}_{l\in\N}$, to an $f_{x_0}$-minimizing solution $\tx$ follows by similar arguments in the proof of Theorem~\ref{tma}. We just have to substitute $\tau_1(\delta + \gamma_m)$ by $\tau_1\delta$ and dominate $\delta + \gamma_m$ by $\lambda/\tau_2 \delta$ based on \eqref{def4}.
\end{proof}


Define the estimate $\eta_m := D_{\xi^\dagger}(P_m\xdag,\xdag)$. It follows from \cite[Corollary~1.2.5]{ekte} that $\eta_m\rightarrow 0$ whenever $m\rightarrow\infty$, since $P_m\xdag\rightarrow \xdag$.


\begin{theo}[Convergence Rates]
Assume that $\xd$ is a minimizer of the functional in Equation~(\ref{tik1})and the regularization parameter $\alpha = \alpha(\delta,\yd,\gamma_m)$ satisfies the discrepancy principle (\ref{def6}). Then, we have the following estimates
\begin{equation}
\|F(\xd) - \yd\| = \mathcal{O}(\delta + \gamma_m + \eta_m) ~~~\text{ and }~~~ D_{\xi^\dagger}(\xd,\xdag) = \mathcal{O}(\delta + \gamma_m + \eta_m + \phi_m).
\end{equation}
with $\xi^\dagger \in \partial f_{x_0}(\xdag)$.
 \label{mor:cr}
\end{theo}

\begin{proof} The first estimate follows directly by Definition~\ref{def6}. Let Assumption~\ref{ass2} hold. Then, by Definition~\ref{def6} it follows that
$$
\begin{array}{rcl}
 \tau_1^p(\delta + \gamma_m)^p + \alpha f_{x_0}(\xd) 
 &\leq& (\delta + \gamma_m)^p + \alpha f_{x_0}(P_{m}\xdag).
\end{array}
$$
This implies that $f_{x_0}(\xd) \leq f_{x_0}(P_{m}\xdag)$ since $\tau_1 > 1$. It also implies that, for $m$ sufficiently large, since $\delta>0$ is fixed, $f_{x_0}(\xd) \leq f_{x_0}(\xdag)$.

From Assumption~\ref{ass2} and the definition of Bregman distances, we have the following estimates, with $\xi^\dagger \in \partial f_{x_0}(\xdag)$:
\begin{multline}
 D_{\xi}(\xd,\xdag) \leq f_{x_0}(\xd) - f_{x_0}(\xdag) - \langle \xi^\dagger, \xd - \xdag \rangle\\
\leq D_{\xi^\dagger}(P_{m}\xdag,\xdag) + \|\xi^\dagger\|\|P_{m}\xdag-\xdag\| + \beta_1 D_{\xi^\dagger}(\xd,\xdag) + 
\beta_2 \|F(\xd) - F(\xdag)\|\\
\leq \eta_m + \|\xi^\dagger\|\phi_m + \beta_1 D_{\xi^\dagger}(\xd,\xdag) + \beta_2(\tau_2 + 1)(\delta + \gamma_m).
\end{multline}
In other words,
\begin{equation}
D_{\xi^\dagger}(\xd,\xdag) \leq \displaystyle\frac{\eta_m + \|\xi^\dagger\|\phi_m + \beta_2(\tau_2 + 1)(\delta + \gamma_m)}{1 - \beta_1}.
\label{convrates2}
\end{equation}
\end{proof}

If $m$ satisfies \eqref{def4}, we have the following corollary:

\begin{cor}[Convergence Rates]
 Assume that Assumption~\ref{ass2} holds true, $m\in\N$ satisfy \eqref{def4} and $\alpha>0$ is chosen through Definition~\ref{def6}. Then, we have the following convergence rates:
\begin{equation}
 \|F(\xd) - \yd\| = \mathcal{O}(\delta)
\end{equation}
and
\begin{equation}
D_{\xi^\dagger}(\xd,\xdag) \leq \displaystyle\frac{(\lambda\delta)^p}{\tau_2^p(\alpha-\alpha\beta_1)} + \displaystyle\frac{\alpha\beta_2^p}{2-2\beta_1} + \displaystyle\frac{\eta_m + \beta_2\delta + \|\xi^\dagger\|\phi_m}{1-\beta_1}.
\label{est3}
\end{equation}
\label{convrates}
\end{cor}

This is an immediate result from the previous theorem.

\begin{rem}
 If $f_{x_0}$ is $q$-coercive with respect to the norm of $X$ and $q=p$, then \eqref{est3} implies that
\begin{equation}
 \|\xd - \xdag\| \leq \left(\frac{1}{\zeta}\displaystyle\frac{(\lambda\delta)^p}{\tau_2^p(\alpha-\alpha\beta_1)} + \displaystyle\frac{\alpha\beta_2^p}{2-2\beta_1} + \frac{1}{\zeta}\displaystyle\frac{\eta_m + \beta_2\delta + \|\xi^\dagger\|\displaystyle\left(\frac{\eta_m}{\zeta}\right)^{\frac{1}{q}}}{1-\beta_1}\right)^{\frac{1}{q}}
\label{est4}
\end{equation}
\end{rem}

\begin{rem}
Note that in the proofs of the above results we have assumed the existence of the abstract quantity
$$
\gamma_m = \|F(\xd) - F(P_m \xd)\|,
$$
which is obviously unknown. However, for some classes of operators $F$, it is possible to estimate an upper bound for this quantity in terms of $I-P_m$, where $I: X\longrightarrow X$ is the identity operator. This is the case, for example, of uniformly H\"older continuous operators:
$$
\|F(x)-F(x^\prime)\| \leq C\|x-x^\prime\|^l,
$$
for every $x,x^\prime \in \domainf$. Thus, $\gamma_m \leq C\|I-P_m\|^l\|\xdag\|^l$ and there exists a constant $K>0$ such that $\|\xdag\|^l\leq K$. In this case, we can choose $m$ sufficiently large such that the inequality below is satisfied:
$$
CK\|I-P_m\|^l \leq \left(\displaystyle\frac{\lambda}{\tau_2}\right)\delta.
$$
\end{rem}

\end{document}